\documentclass[10pt]{amsart}
\usepackage{amsxtra, amsfonts, amsmath, amsthm, amstext, amssymb, amscd, mathrsfs, verbatim, color}
\usepackage{threeparttable}
\usepackage{mathtools}
\usepackage[ansinew]{inputenc}\usepackage[T1]{fontenc}
\theoremstyle{plain}
\usepackage{amsthm,amssymb}

\newtheorem{theorem}{Theorem}[section]
\newtheorem{lemma}[theorem]{Lemma}
\newtheorem{definition-theorem}[theorem]{Definition-Theorem}
\newtheorem{proposition}[theorem]{Proposition}
\newtheorem{corollary}[theorem]{Corollary}

\newtheorem{conjecture}[theorem]{Conjecture}

\theoremstyle{definition}
\newtheorem{definition}[theorem]{Definition}
\newtheorem{example}[theorem]{Example}
\newtheorem{remark}[theorem]{Remark}
\newtheorem{notation}[theorem]{Notation}

\newcommand \bth[1] { \begin{theorem}\label{t#1} }
\newcommand \ble[1] { \begin{lemma}\label{l#1} }

\newcommand \bpr[1] { \begin{proposition}\label{p#1} }
\newcommand \bco[1] { \begin{corollary}\label{c#1} }
\newcommand \bde[1] { \begin{definition}\label{d#1}\rm }
\newcommand \bex[1] { \begin{example}\label{e#1}\rm }
\newcommand \bre[1] { \begin{remark}\label{r#1}\rm }

\newcommand \bnota[1] {\begin{notation}\label{n#1}\rm }
\newcommand {\ele} { \end{lemma} }

\newcommand {\epr} { \end{proposition} }
\newcommand {\eco} { \end{corollary} }
\newcommand {\ede} { \end{definition} }
\newcommand {\eex} { \end{example} }
\newcommand {\ere} { \end{remark} }
\newcommand {\enota} { \end{notation} }

\newcommand\pr[2]{\langle{#1},{#2}\rangle}
\def \Id { {\mathrm{Id}} }

\def \diag { {\mathrm{diag}}}

\def \pr { { \mathrm{pr}}}

\DeclareMathOperator \Span { {\mathrm{span}} }

\DeclareMathOperator \card { {\mathrm{card}}}

\DeclareMathOperator \tr { {\mathrm{tr}} }

\DeclareMathOperator \Hom { {\mathrm{Hom}} }
\DeclareMathOperator \Ext { { \mathrm{Ext}}}

\DeclareMathOperator \Ind { {\mathrm{Ind}} }
\DeclareMathOperator \Irr { {\mathrm{Irr}} }
\DeclareMathOperator \Res { {\mathrm{Res}} }

\DeclareMathOperator \rad { {\mathrm{rad}}}

\DeclareMathOperator \sgn { { \mathrm{sgn}}}

\DeclareMathOperator \im { { {\mathrm im}}}

\begin{document}
\setlength{\baselineskip}{1.2\baselineskip}
\title[Twisted Euler-Poincar\'e pairing for graded affine Hecke algebras]
{On a twisted Euler-Poincar\'e pairing for graded affine Hecke algebras}
\author[Kei Yuen Chan]{Kei Yuen Chan}
\address{
Department of Mathematics \\
University of Utah}
\email{chan@math.utah.edu}

\maketitle 

\begin{abstract}
We study a twisted Euler-Poincar\'e pairing for graded affine Hecke algebras, and give a precise connection to the twisted elliptic pairing of Weyl groups defined by Ciubotaru-He \cite{CH}. The $\Ext$-groups for an interesting class of parabolically induced modules are also studied in a connection with the twisted Euler-Poincar\'e pairing. We also study a certain space of graded Hecke algebra modules which equips with the twisted Euler-Poincar\'e pairing as an inner product.
\end{abstract}

\section{Introduction}

This paper studies a twisted Euler-Poincar\'e pairing on the space of virtual representations for the graded affine Hecke algebra. This twisted pairing is motivated from the twisted elliptic pairing of Weyl group recently developed by Ciubotaru-He \cite{CH}, and we give a precise relations between these two pairings. In the same spirit as the Euler-Poincar\'e pairing of $p$-adic groups by Schneider-Stuhler \cite{SS} and others, an appropriate subspace of the virtual representations for the graded Hecke algebra is equipped with the twisted Euler-Poincar\'e pairing as an inner product. We shall discuss those twisted elliptic spaces defined by the twisted Euler-Poincar\'e pairing (based on several previous work by others \cite{Ci}, \cite{CH0}, \cite{CH}, \cite{Re} and \cite{OS2}). 

In more detail, let $(R, V, R^{\vee}, V^{\vee})$ be a root data of a crystallographic type (Section \ref{ss basic notation}) and let $W$ be the finite reflection group acting on $R$. Let $\Delta$ be the set of simple roots. Let $\delta$ be an involution on the root system with $\delta(\Delta)=\Delta$. Then $\delta$ induces an involution on $W$ which is still denoted by $\delta$. A recent paper of Ciubotaru-He \cite{CH}  defined the $\delta$-twisted elliptic pairings on the representations $U$ and $U'$ of $W \rtimes \langle \delta \rangle$ as:
\[  \langle U, U' \rangle^{\delta-\mathrm{ellip}, V}_W = \frac{1}{|W|}\sum_{w \in W} \tr_U(w \delta)\overline{\tr_{U'}(w\delta)}\mathrm{det}_{V}(1-w\delta) ,
\]
where $\tr$ is the trace of $w$ acting on $U$ or $U'$. This twisted elliptic pairing is closely related to the Lusztig-Shoji algorithm.

When $\delta=\Id$, the pairing coincides with the one defined by Reeder \cite{Re}. Suggested by Arthur \cite{Ar} and verified by Reeder \cite{Re}, a precise relation between the Euler-Poincar\'e pairing for $p$-adic groups and an elliptic pairing of Weyl groups was established. The goal of this paper is to study an analogue of the Euler-Poincar\'e pairing relating to the $\delta$-twisted elliptic pairing considered by Ciubotaru-He. Our work is done in the level of graded affine Hecke algebra, which was introduced by Lusztig in \cite{Lu} for the study of representations of $p$-adic groups and Iwahori-Hecke algebras. 

Let $\mathbb{H}$ be the graded affine Hecke algebra associated to a crystallographic root system $(R, V, R^{\vee}, V^{\vee})$ and a parameter function $k$ (Definition \ref{def graded affine}). The action of $\delta$ can be extended to the Weyl group, and then extended to $\mathbb{H}$. For $\mathbb{H} \rtimes \langle \delta \rangle$-modules $X$ and $Y$, we define the $\delta$-twisted Euler-Poincar\'e pairing on $X$ and $Y$ (regarded as $\mathbb{H}$-modules):
\[  \mathrm{EP}_{\mathbb{H}}^{\delta}(X, Y) = \sum_{i}(-1)^i \mathrm{trace}( \delta^*: \Ext^i_{\mathbb{H}}(X, Y) \rightarrow  \Ext_{\mathbb{H}}^i(X, Y) ),\]
where $\Ext$-groups are taken in the category of $\mathbb{H}$-modules. Here $\delta^*$ is a natural map induced from the action of $\delta$ on $X$ and $Y$. Our first main result is the following:
\begin{theorem} \label{thm inner formula} (Proposition \ref{thm euler poincare}, Theorem \ref{thm twisted ext})
Suppose $\delta$ induces an inner automorphism on $W$ (equivalently $\delta=\Id$ or $\delta$ arises from the longest element in the Weyl group (see \ref{eqn involution})). For any finite dimensional $\mathbb{H} \rtimes \langle \delta \rangle$-modules $X$ and $X'$, 
\[  \mathrm{EP}^{\delta}_{\mathbb{H}}(X, X') = \langle \Res_W X, \Res_W X' \rangle^{\delta-\mathrm{ellip}, V}_W,\]
where $\Res_W$ is the restriction to the $W$-representation.
\end{theorem}
\noindent
Theorem \ref{thm inner formula} for $\delta=\Id$ was established by Reeder \cite{Re} for equal parameter cases, and was independently proved by Opdam-Solleveld \cite{OS} for arbitrary parameters (in different settings). Nevertheless, our approach in proving Theorem \ref{thm inner formula} is independent from their work, and is self-contained. We remark our proof of Theorem \ref{thm inner formula} also holds for non-crystallographic cases, and the consequences for those cases will be considered elsewhere.

Our study begins with the construction of an explicit projective resolution on $\mathbb{H}$-modules.  The idea of the construction came from the standard Koszul resolution. A remarkable point is that taking the $\Hom$-functor on the resolution, the $\Hom$-spaces between $\mathbb{H}$-modules are turned into $\Hom$-spaces between Weyl group representations via Frobenious reciprocity, which is also essential in the proof of Theorem \ref{thm inner formula}.

When $\delta=\Id$, the pairing defines an inner product on a subspace of the $\mathbb{H}$-representation ring. This space has been known and studied in \cite{Re} and \cite{OS}. Our focus of the remaining discussion will be on the case that $\delta$ is the automorphism $\theta$ arising from the longest element in the Weyl group (see (\ref{eqn involution})). Similar to the case for $\theta=\Id$, an appropriate subspace of the representation ring of $\mathbb{H}$ is equipped with $\mathrm{EP}^{\theta}_{\mathbb{H}}$ as an inner product. We call such space to be $\theta$-twisted elliptic as an analogue to the case in $p$-adic groups considered by Schneider-Stuhler \cite{SS}. Such $\theta$-twisted elliptic space can also be regarded as the elliptic representaion space of $\mathbb{H} \rtimes \langle \theta \rangle$. We shall describe those $\theta$-twisted elliptic space in the next paragraph.

Let $\mathcal N_{\mathrm{sol}}$ be the set of nilpotent elements which have a solvable centralizer in the related Lie algebra to the root system. This set naturally arises from the study of the spin representations of Weyl groups as well as the Dirac cohomology for the graded affine Hecke algebra (\cite{CH}, \cite{BCT}, \cite{Ch1}, \cite{Ci}). In particular, the work of Ciubotaru-He \cite{CH} implies that in the case of equal parameters, the $\theta$-twisted elliptic representation space of $\mathbb{H}$ is spanned by tempered modules which correspond to a nilpotent element in $\mathcal N_{\mathrm{sol}}$ under the Kazhdan-Lusztig parametrization (Theorem \ref{cor twisted elliptic}). For the simplicity later, we shall call those tempered modules to be solvable.

Those solvable tempered modules can be divided into three classes. The first ones are those (ordinary) elliptic tempered modules (in the sense of Reeder \cite{Re}). The second ones are those irreducible non-elliptic tempered modules which are not properly parabolically induced. This happens for the type $D_n$ for $n$ odd and $n \geq 9$ (see Remark \ref{rmk non ellip rig}). The third ones are certain irreducible, tempered and parabolically induced modules. It turns out that those irreducible tempered module in the third class can be characterized by a simple condition on the parabolic subalgebra which it is induced from. Those classes of modules are called rigid modules in Definition \ref{def rigid} and Proposition \ref{prop sol rigid}. A deeper reasoning for such condition indeed comes from the Plancherel measure and $R$-groups (in the sense of Opdam \cite{Op} and \cite{DO2} respectively). The study related to those harmonic analysis interpretations on solvable tempered modules will be carried out elsewhere \cite{Ch2} (also see Remark \ref{rmk general formulation}). 


Our second part of the paper is to study the $\Ext$-groups on the rigid modules in Definition \ref{def rigid}. (See Remark \ref{rmk degenerate limits} for more comments on the terminology.)  As mentioned above, rigid modules provide most examples of solvable tempered modules which are not elliptic. In other words, they lie in the radical of the (ordinary) Euler-Poincar\'e pairing, but not in the radical of the twisted Euler-Poincar\'e pairing. Then it is natural to ask how those rigid modules behave differently under the two pairings via a study of the $\Ext$-groups and the $\theta^*$-action.


Another main result in this paper is Theorem \ref{thm theta action intro} below. 
\begin{theorem} (Theorem \ref{thm theta action}) \label{thm theta action intro}
Let $\mathbb{H}$ be the graded affine Hecke algebra associated to a crystallographic root system and a parameter function $k$ (Definition \ref{def graded affine}). Let $X$ be a rigid of discrete series of $\mathbb{H}$ (Definition \ref{def rigid}). Then 
\[  \dim \Ext_{\mathbb{H}}^i(X,X)=\left( \begin{array}{rl} r \\ i \end{array} \right)=\frac{r!}{(r-i)!i!},  \quad \mbox{ for $ i \leq r$}
\]
for some fixed $r$ (which is described precisely in Theorem \ref{thm theta action}). Furthermore $\theta^*$  acts on $\Ext_{\mathbb{H}}^i(X, X)$ by the multiplication of a scalar of $(-1)^i$.
\end{theorem}

We remark that our computation of $\Ext$-groups in Theorem \ref{thm theta action intro} essentially uses the $\Ext$-groups for discrete series from the work of Delorme-Opdam \cite{DO} and Opdam-Solleveld \cite{OS}. Apart from the deep analytic result from \cite{DO} and \cite{OS}, the main tool of our computation is the projective resolution developed in Section \ref{s koszul type resolution} with some careful analysis on the structure of rigid modules. It is possible to apply our techniques to other tempered modules, but results obtained by current approach is more complete for those rigid modules.

The approach used in this paper to study $\Ext$-groups differs from the one used by Adler-Prasad \cite{AP} for $p$-adic groups and the one by Opdam-Solleveld \cite{OS2} for affine Hecke algebras, and so we hope our study provides another perspective on the extensions of representations. Our approach should also be applicable for the study of the graded Hecke algebra of a noncrystallographic type and other similar algebraic structure such as the degenerate affine Hecke-Clifford algebra.

We briefly outline the organization of this paper. Section \ref{s involution hermitian} is to define and review several important objects such as the map $\theta$, graded affine Hecke algebras and tempered modules. In Section \ref{s koszul type resolution}, we construct an explicit projective resolution of an $\mathbb{H}$-module, which is the main tool in this paper. In Section \ref{s twisted EP}, we define the twisted Euler-Poincar'e pairing and prove Theorem \ref{thm inner formula}. Section \ref{s theta action} is devoted to compute the $\theta^*$ action on some $\mathrm{Ext}$-groups of certain modules. Section \ref{s solvable induced} is to study and describe the twisted elliptic space in terms of the Kazhdan-Lusztig model.
\noindent



\subsection{Acknowledgement}

The author would like to thank his PhD advisor Peter Trapa for initiating this research and having many useful discussions. He also thanks Peter Trapa  for pointing out the definition of the twisted Euler-Poincar\'e pairing and providing his idea on Theorem \ref{thm twisted ext}. He is also grateful for Dan Ciubotaru and Xuhua He for useful discussions on elliptic modules and their papers \cite{CH0, CH}. He would also like to thank Marteen Solleveld for providing many useful and detailed suggestions in an earlier version of this paper.
\section{Preliminaries} \label{s involution hermitian}

\subsection{Root systems and basic notations} \label{ss basic notation}
Let $R$ be a reduced root system of a crystallographic type. Let $\Delta$ be a fixed choice of simple roots in $R$. Then $\Delta$ determines the set of positive roots $R^+$. Let $W$ be the finite reflection group of $R$. Let $V_0'$ be the real space spanned by $\Delta$ and let $V_0$ be a real vector space containing $V_0'$ as a subspace. For any $\alpha \in \Delta$, let $s_{\alpha}$ be the simple reflection in $W$ associated to $\alpha$ (i.e. $\alpha \in V_0$ is in the $-1$-eigenspace of $s_{\alpha}$). For $\alpha \in R$, let $\alpha^{\vee} \in \Hom_{\mathbb{R}}(V_0, \mathbb{R})$ such that
\[   s_{\alpha}(v) = v-\langle v, \alpha^{\vee} \rangle \alpha, \]
where $\langle v, \alpha^{\vee} \rangle=\alpha^{\vee}(v)$. Let $R^{\vee} \subset \Hom_{\mathbb{R}}(V_0, \mathbb{R})$ be the collection of all $\alpha^{\vee}$. Let $V_0^{\vee}= \Hom_{\mathbb{R}}(V_0, \mathbb{R})$.

By extending the scalars, let $V=\mathbb{C} \otimes_{\mathbb{R}} V_0$ and let $V^{\vee}=\mathbb{C} \otimes_{\mathbb{R}} V_0^{\vee}$. We call $(R, V, R^{\vee}, V^{\vee})$ to be a root datum.

For any subset $J$ of $\Delta$, define $V_J$ to be the complex subspace of $V$ spanned by simple roots in $J$. Let $R_J = V_J \cap R$. Let $R_J^{\vee}= \left\{ \alpha^{\vee} \in R^{\vee}: \alpha \in R_J \right\}$. Let $V_J^{\vee}$ be the subspace of $V^{\vee}$ spanned by the coroots in $R_J^{\vee}$. Let $W_J$ be the subgroup of $W$ generated by the elements $s_{\alpha}$ for $\alpha \in J$. Define 
\[  V_J^{\bot} = \left\{  v \in V : \langle v, v_1^{\vee} \rangle =0 \quad \mbox{ for all $v_1^{\vee} \in V_J^{\vee} $ } \right\}, \]
and
\[ (V_J^{\vee})^{\bot} = \left\{ v^{\vee} \in V^{\vee} : \langle v_1, v^{\vee} \rangle =0 \quad \mbox{ for all $v_1 \in V_J $ } \right\}.\]

Let $J \subset \Delta$. Let $w_{0, J}$ be the longest element in $W_J$. When $J=\Delta$, we simply write $w_{0}$ for $w_{0, \Delta}$. Let $W^J$ be the set of minimal representatives in the cosets in $W/W_J$. Let $w_0^J$ be the longest element in $W^J$.

\subsection{Graded affine Hecke algebras}

Let $k: \Delta \rightarrow \mathbb{R}$ be a parameter function such that $k(\alpha)=k(\alpha')$ if $\alpha$ and $\alpha'$ are in the same $W$-orbit. We shall simply write $k_{\alpha}$ for $k(\alpha)$.

\begin{definition} \label{def graded affine} \cite[Section 4]{Lu}
The graded affine Hecke algebra $\mathbb{H}=\mathbb{H}_W$ associated to a root data $(R, V, R^{\vee}, V^{\vee})$ and a parameter function $k$ is an associative algebra with an unit over $\mathbb{C}$ generated by the symbols $\left\{ t_w :w \in W \right\}$ and $\left\{ f_w: w \in V \right\}$ satisfying the following relations:
\begin{enumerate}
\item[(1)] The map $w  \mapsto t_w$ from $\mathbb{C}[W]=\oplus_{w\in W} \mathbb{C}w  \rightarrow \mathbb{H }$ is an algebra injection,
\item[(2)] The map $v \mapsto f_v$ from $S(V) \rightarrow \mathbb{H}$ is an algebra injection,
\end{enumerate}
For simplicity, we shall simply write $v$ for $f_v$ from now on.
\begin{enumerate}
\item[(3)] the generators satisfy the following relation:
\[    t_{s_{\alpha}}v-s_{\alpha}(v)t_{s_{\alpha}}=k_{\alpha}\langle v, \alpha^{\vee} \rangle .\]
\end{enumerate}

\end{definition}

\begin{notation}
Let $J \subset \Delta$. Define $\mathbb{H}_J$ to be the subalgebra of $\mathbb{H}$ generated by all $v \in V$ and $t_w$ ($w \in W_J$). We also define $\overline{\mathbb{H}}_J$ to be the subalgebra of $\mathbb{H}$ generated by all $v \in V_J$ and $t_w$ ($w \in W_J$). Here $V_J$ and $W_J$ is defined in Section \ref{ss basic notation}. Note that $\mathbb{H}_J$ decomposes as
\[  \mathbb{H}_J = \overline{\mathbb{H}}_J  \otimes S(V_J^{\bot}) .
\]
Note that $\mathbb{H}_J$ is the graded affine Hecke algebra associated to the root data $(R, V_0, R^{\vee}, V_0^{\vee})$ and $\overline{\mathbb{H}}_J$ is the graded affine Hecke algebra associated to the root data $(R, V_J, R^{\vee}, V^{\vee}_J)$.
\end{notation}

\begin{notation}
According to (1) and (2), we shall view $\mathbb{C}[W]$ and $S(V)$ as the natural subalgebras of $\mathbb{H}$. For an $\mathbb{H}$-module $X$ (resp. $\mathbb{H}_J$-module $X$ with $J \subset \Delta$), denote $\Res_WX$ (resp. $\Res_{W_J}X$) be the restriction of $X$ to a $\mathbb{C}[W]$-module (resp. $\mathbb{C}[W_J]$-module). $\Res_{\mathbb{H}_J}$ and $\Res_{\overline{\mathbb{H}}_J}$ are defined similarly for $\mathbb{H}$-modules.

\end{notation}

For $v \in V$, we define the following element in $\mathbb{H}$:
\begin{eqnarray} \label{eqn v titlde}
    \widetilde{v}&= v- \frac{1}{2} \sum_{\alpha \in R^+} c_{\alpha} \langle v, \alpha^{\vee} \rangle s_{\alpha} .
\end{eqnarray}
This element is used in \cite{BCT} for the study of the Dirac cohomology for graded affine Hecke algebras.

\begin{lemma} \label{lem tilde element}
 For any $w \in W$ and $v \in V$, $t_w\widetilde{v}=\widetilde{w(v)}t_w$.
\end{lemma}

\begin{proof}
It suffices to show for the case that $w$ is a simple reflection $s_{\beta} \in W$.
\begin{eqnarray*}
 t_{s_{\beta}}\widetilde{v} &=&  t_{s_{\beta}}\left( v- \frac{1}{2} \sum_{\alpha \in R^+} k_{\alpha} \langle v, \alpha^{\vee} \rangle t_{s_{\alpha}} \right)  \\
                             &=&  s_{\beta}(v)t_{s_{\beta}}+k_{\beta}\langle v, \beta^{\vee}\rangle -\frac{1}{2}k_{\beta}\langle v, \beta^{\vee} \rangle -\frac{1}{2}\sum_{\alpha \in R^+\setminus \left\{\beta\right\} }  k_{\alpha} \langle  v, \alpha^{\vee} \rangle t_{s_{\beta}(\alpha)}  \\
														&=&   s_{\beta}(v)t_{s_{\beta}} -\frac{1}{2}k_{\beta}\langle v, s_{\beta}(\beta^{\vee})\rangle  -\frac{1}{2}\sum_{\alpha \in R^+\setminus \left\{\beta\right\} }  k_{\alpha} \langle  v, s_{\beta}(\alpha^{\vee}) \rangle t_{\alpha} t_{s_{\beta}} \\
														&=& s_{\beta}(v)t_{s_{\beta}}-\frac{1}{2} \sum_{\alpha \in R^+} k_{\alpha} \langle s_{\beta}(v), \alpha^{\vee} \rangle t_{\alpha}t_{s_{\beta}}  \\
														&=& \widetilde{s_{\beta}(v)}
\end{eqnarray*}

\end{proof}

\subsection{Central characters of $\mathbb{H}$}

The center of $\mathbb{H}$ can be explicitly described as below.

\begin{proposition} \label{prop center of hecke algebra} \cite[Proposition 4.5]{Lu}
The center of $\mathbb{H}$ is equal to $S(V)^W$, where $S(V)^W$ is the set of the $W$-invariant polynomials in $S(V)$
\end{proposition}

\begin{definition}
Let $Z(\mathbb{H})$ be the center of $\mathbb{H}$. The central character of an irreducible $\mathbb{H}$-module $X$ is the map $\chi: Z(\mathbb{H}) \rightarrow \mathbb{C}$ such that $\chi(z)$ is the scalar that $z$ acts on $X$.
\end{definition}

According to Proposition \ref{prop center of hecke algebra}, the central character $\chi$ can be parametrized by the $W$-orbits $[v]$ in $V$ such that 
\[ \chi(z)=z(v),
\]
where $v$ is a representative of the $W$-orbit $[v]$ and $z(v)$ is regarded as the value of the polynomial $z$ evaluated at $v$.

\subsection{$*$-operation and $*$-Hermitian modules}

We first define an anti-involutive $*$-operation which naturally comes from the $p$-adic groups as follow:
\[  t_w^*=t_w^{-1} \quad \mbox{for $w \in W$}, \quad v^*=-t_{w_0}\overline{w_0(v)}t_{w_0}^{-1}=\overline{-v+\frac{1}{2}\sum_{\alpha \in R^+} \langle v, \alpha^{\vee} \rangle t_{s_{\alpha}}} .\]
Here $\overline{h}$ denotes the complex conjugation on $h$. 

\begin{definition}
Let $X$ be an $\mathbb{H}$-module. A function $f: X \rightarrow \mathbb{C}$ is said to be conjugate-linear if $f(\lambda x_1+x_2)=\overline{\lambda}f(x_1)+f(x_2)$ for all $\lambda \in \mathbb{C}$ and $x_1, x_2 \in X$. The {\it $*$-Hermitian dual} of $X$, denoted $X^*$, is the space of all the conjugate-linear functions $f: X \rightarrow \mathbb{C}$
equipped with the $\mathbb{H}$-action given by
\[   (h.f)(x)=f(h^*.x) \quad \mbox{ for all $x \in X$} . \]
It is straightforward to verify that the above $\mathbb{H}$-action is well-defined. An $\mathbb{H}$-module $X$ is said to be {\it $*$-Hermitian} if $X$ is isomorphic  to its Hermitian dual, or equivalently there exists a non-degenerate Hermitian form on $X$ such that $\langle h.x_1, x_2 \rangle =\langle x_1, h^*.x_2 \rangle$ for all $h \in \mathbb{H}$ and $x_1, x_2 \in X$.  

We say that $X$ is {\it $*$-unitary} if there exists a non-degenerate and positive-definite Hermitian form on $X$ such that $\langle h.x_1, x_2 \rangle=\langle x_1, h^*.x_2 \rangle$ for all $h \in \mathbb{H}$ and $x_1, x_2 \in X$. 
\end{definition}




\subsection{$\theta$-action} \label{ss theta action}

Let $\theta$ be an involution on $\mathbb{H}$ characterized by
\begin{eqnarray}
\label{eqn involution} 
\theta(v)=-w_0(v) \mbox{ for any $v \in V$}, \mbox{ and } \quad \theta(t_w)=t_{w_0ww_0^{-1}} \mbox{ for any $w \in W$} ,
\end{eqnarray}
where $w_0$ acts on $v$ as the reflection representation of $W$.

\begin{lemma}
For any $v \in V$, $\theta(\widetilde{v})=\widetilde{\theta(v)}$.
\end{lemma}

\begin{proof}
This follows from a straightforward computation.
\end{proof}

\begin{definition}
For an $\mathbb{H}$-module $X$, define $X^{\theta}$ to be the $\mathbb{H}$-module such that $X^{\theta}$ is isomorphic to $X$ as vector spaces and the $\mathbb{H}$-action is determined by:
\[   \pi_{X^{\theta}}(h)x=\pi_X(\theta(h))x ,\]
where $\pi_X$ and $\pi_{X^{\theta}}$ are the maps defining the action of $\mathbb{H}$ on $X$ and $X^{\theta}$ respectively.
\end{definition}

\begin{definition}
Let $J \subset \Delta$. For an $\mathbb{H}_J$-module $X$, we say $\gamma \in V^{\vee}$ is a {\it weight} of $X$ if there exists a non-zero $x \in X$ such that $(v-\gamma(v))^kx=0$ for all $v \in V$ and for some positive integer $k$. We call such $x$ to be the {\it generalized weight vector} of $\gamma$.

\end{definition}

\begin{proposition} \label{prop herm}
Let $X$ be an irreducible $\mathbb{H}$-module with a real central character. Assume that $X$ satisfy one of the following conditions:
\begin{enumerate}
\item the central character of $X$ is non-zero,
\item the parameter function $k$ is identically equal to zero,
\item $k_{\alpha}\neq 0$ for all $\alpha \in \Delta$.
\end{enumerate}
Then $X^{\theta}$ is the Hermitian dual of $X$.
\end{proposition}
\begin{proof}
We sketch the proof. Let $x_{\gamma}$ be a generalized weight vector of $X$ of a weight $\gamma \in V^{\vee}$. Then for sufficiently large $k$ and $v \in V_0$,
\[ ((v-\theta(\gamma)(v))^k.f)(t_{w_0}.x_{\gamma})=f(t_{w_0}(\theta(v)-\theta(\gamma)(v))^k.x_{\gamma}))=0.
\]
Hence $\overline{\theta(\gamma)}=\theta(\gamma)$ is a weight of the Hermitian dual of $X$. Then  have the same weights.

If $X$ satisfies (1), then the arguments in the proof of \cite[Proposition 4.3.1]{BC} (also see \cite[Theorem 5.5]{EM}) implies that $X$ and $X^{\theta}$ are isomorphic. We now assume (1) does not hold for $X$. Then the central character of $X$ is zero. If $X$ satisfies (2), then the restriction of $X$ to $\mathbb{C}[W]$ is an irreducble $W$-representation. Then it is easy to show that the Hermitian dual of $X$ and $X^{\theta}$ are isomorphic. We now assume $X$ satisfies (3). Then by \cite[Theorem 1.3]{OS0} or \cite[Proposition 2.9]{KR}, $\mathrm{Ind}^{\mathbb{H}}_{S(V)}\mathbb{C}_0$ is irreducible and hence there is only one irreducible $\mathbb{H}$-module with the central character $0$. This implies the Hermitina dual of $X$ and $X^{\theta}$ are isomorphic.


\end{proof}

\begin{remark}
We believe that Proposition \ref{prop herm} is true for all the $\mathbb{H}$-modules with a real central character (without assuming any one of the three conditions in the propsoition). An evidence is that the Hermitian dual of $X$ and $X^{\theta}$ have the same $S(V)$ and $\mathbb{C}[W]$ module structure. However, the author does not succeed to find a simple proof. For the purpose of this paper, modules satisfying any one of the three conditions suffice.

\end{remark}

\begin{corollary} \label{cor hermitian criteria}
Let $X$ be an irreducible $\mathbb{H}$-module with a real central character. Assume $X$ satisfies any one of the three conditions in Proposition \ref{prop herm}. Then $X$ is a $*$-Hermitian $\mathbb{H}$-module if and only if $X$ and $X^{\theta}$ are isomorphic.
\end{corollary}



\subsection{Tempered modules and discrete series} 

Tempered modules and discrete series will be studied in Section \ref{s theta action} and \ref{s solvable induced}. They provide the main examples of $\mathbb{H}$-modules $X$ with the property $X^{\theta} = X$. 

\begin{definition}
\label{def tempered}
Recall that $\mathbb{H}$ is associated to the root data $(R, V, R^{\vee}, V^{\vee})$. An $\mathbb{H}$-module $X$ is said to be {\it tempered} if for any weight $\gamma \in V^{\vee}$ of $X$, $\mathrm{Re}\langle \omega_{\alpha}, \gamma \rangle \leq 0$ for any fundamental weight $\omega_{\alpha}$ in $V$. Here $\mathrm{Re}(a)$ denotes the real part of a complex number. 

An $\mathbb{H}$-module is said to be a {\it discrete series} if $X$ is tempered and all the inequalities in the definition of tempered modules are strict. 

\end{definition}

\begin{theorem} \cite[Theorem 7.2]{So} \label{thm disc unitary}
All irreducible discrete series has a real central character and are $*$-unitary. 
\end{theorem}
\begin{notation} \label{not induced modules}
Let $\Xi$ be the set of triples $(J, U, \nu)$ such that $J \subset \Delta$, $U$ is a $\overline{\mathbb{H}}_J$-discrete series, and $\nu \in V_J^{\vee}$. For any $(J, U, \nu) \in \Xi$, denote $X(J, U, \nu)$ to be the parabolically induced module $\Ind_{\mathbb{H}_J}^{\mathbb{H}}(U \otimes \mathbb{C}_{\nu}) :=\mathbb{H} \otimes_{\mathbb{H}_J}(U \otimes \mathbb{C}_{\nu})$. When $\nu=0$, we shall simply write $X(J, U)$ instead of $X(J, U, 0)$. We indeed consider $\nu=0$ most of time in this paper. We call $X(J,U)$ to have a real central character (c.f. Theorem \ref{thm disc unitary}).
\end{notation}

\begin{proposition} \label{prop Hermitian form} \cite[Corollary 1.4]{BM}
Let $(J, U, \nu) \in \Xi$. Then there exists a non-degenerate positive-definite $*$-Hermitian form $\langle ,  \rangle$ on $X(J, U, \nu)$ i.e. $\langle h.x, x' \rangle = \langle x, h^*.x' \rangle$. In particular, $X(J, U, \nu)$ is $*$-unitary.

\end{proposition}

\begin{proof}
This is \cite[Corollary 1.4]{BM}. Since $U$ is an irreducible $\mathbb{H}_J$-discrete series, Theorem \ref{thm disc unitary} implies that there exists a non-degenerate $*$-Hermitian form $\langle , \rangle_{J}$ on $U$. Define a projection map $\pr: \mathbb{H} \rightarrow \mathbb{H}_J$ as follow: for $h \in \mathbb{H}$, $h$ can be uniquely written as the form $\sum_{w \in W^J} t_w h_w$, where $h_w \in \mathbb{H}_J$. Then $\pr$ is defined as $\pr(h)=h_e$, where $e$ corresponds to the trivial coset in $W/W_J$. Define the non-degenerate form $\langle ,  \rangle$ on $X(J, U)$ as 
\[   \langle h_1 \otimes u_1, h_2 \otimes u_2 \rangle = \langle u_1, \pr(h_1^*h_2)u_2 \rangle_J .\]
It remains to verify $\langle , \rangle $ satisfies the desired properties. 
\end{proof}

We shall use the following result later:

\begin{corollary}
Let $(J, U, 0) \in \Xi$. Suppose $X(J, U)$ satisfy one of the three conditions in Proposition \ref{prop herm}. Then $X(J,U)$ is isomorphic to $X(J, U)^{\theta}$ as $\mathbb{H}$-modules.

\end{corollary}

\begin{proof}
By Proposition \ref{prop Hermitian form}, $X(J, U)$ is the direct sum of irreducible $*$-Hermitian modules. Then the statement is a consequence of Proposition \ref{prop Hermitian form} and Corollary \ref{cor hermitian criteria}.
\end{proof}

\subsection{$\mathrm{Ext}_{\mathbb{H}}$-groups}
The following result about $\Ext_{\mathbb{H}}$-groups will be used several times later. Here $\Ext_{\mathbb{H}}$-groups are taken in the category of $\mathbb{H}$-modules.

\begin{theorem} \label{thm char ext 0}
Let $X$ and $Y$ be $\mathbb{H}$-modules. Then if $X$ and $Y$ have distinct central characters, then $\Ext^i_{\mathbb{H}}(X, Y)=0$ for all $i$. 

\end{theorem}

\begin{proof}
See for example \cite[Theorem I. 4.1]{BW}, whose proof can be modified to our setting.
\end{proof}

\section{A Koszul type resolution on $\mathbb{H}$-modules} \label{s koszul type resolution}

We keep using the notation in Section \ref{s involution hermitian}.

\subsection{Koszul-type resolution on $\mathbb{H}$-modules}


Let $X$ be an $\mathbb{H}$-module. Define a sequence of $\mathbb{H}$-module maps $d_i$ as follows:
\begin{align} \label{eqn projective resolution} 0 \rightarrow \mathbb{H} \otimes_{\mathbb{C}[W]}(\Res_WX \otimes \wedge^n V)  \stackrel{d_n}{\rightarrow} \ldots  \stackrel{d_{i+1}}{\rightarrow} \mathbb{H} \otimes_{\mathbb{C}[W]}(\Res_WX \otimes \wedge^i V) & \stackrel{d_i}{\rightarrow} \ldots  \stackrel{d_1}{\rightarrow} \mathbb{H} \otimes_{\mathbb{C}[W]}\Res_WX \stackrel{d_0}{\rightarrow}  X \rightarrow 0
\end{align}
such that $d_0: \mathbb{H} \otimes X \rightarrow X$ given by
\[  d_0(h \otimes x)= h.x
\]
and for $i \geq 1$, $d_i: \mathbb{H} \otimes_{\mathbb{C}[W]} (\Res_WX \otimes \wedge^i V) \rightarrow  \mathbb{H} \otimes_{\mathbb{C}[W]} (\Res_WX \otimes \wedge^{i-1} V)$ given by
 \begin{align}  
&d_i(h\otimes (x \otimes v_1 \wedge\ldots \wedge v_i))  \\
= & \sum_{j=0}^i (-1)^j(h v_j \otimes x \otimes v_1 \wedge \ldots \wedge \hat{v}_j \wedge \ldots \wedge v_i -h \otimes v_j. x \otimes v_1 \wedge \ldots \wedge \hat{v}_j \wedge \ldots \wedge v_i ) .
\end{align}

\begin{proposition} \label{prop well defined d}
The above $d_i$ are well-defined maps and $d^2=0$ i.e. (\ref{eqn projective resolution}) is a well-defined complex.
\end{proposition}

\begin{proof}
We proceed by an induction on $i$. It is easy to see that $d_0$ is well-defined. We now assume $i \geq 1$. To show $d_i$ is independent of the choice of a representative in $\mathbb{H} \otimes_{\mathbb{C}[W]}(X \otimes \wedge^i V)$, it suffices to show 
\begin{align} \label{eqn di map}
    d_i(t_w \otimes (x \otimes v_1\wedge \ldots \wedge v_i) =d_i(1 \otimes (t_w.x \otimes w(v_1)\wedge \ldots \wedge w(v_i))).
\end{align}
For simplicity, set
\begin{eqnarray*}
  P^w &=& d_i(t_w \otimes (x \otimes v_1\wedge \ldots \wedge v_i)  \\
      &=& t_w\sum_{k=0}^i (-1)^{k-1}[v_i \otimes (x \otimes v_1 \wedge \ldots \wedge \widehat{v}_k \wedge \ldots \wedge v_i )-  1 \otimes (v_k.x \otimes v_1\wedge \ldots \wedge \widehat{v}_k \wedge \ldots \wedge v_i )]
\end{eqnarray*}
and
\begin{eqnarray*}
  P_w &=& d_i(1 \otimes (t_w.x \otimes w(v_1)\wedge \ldots \wedge w(v_i))  \\
      &=& \sum_{i=0}^k (-1)^{k-1}w(v_k) \otimes (t_w.x \otimes w(v_1) \wedge \ldots \wedge \widehat{w(v_k)} \wedge \ldots \wedge w(v_i)) \\
			& & \quad - \sum_{i=0}^k (-1)^{k-1}  \otimes (w(v_i).t_w.x \otimes w(v_1)\wedge \ldots \wedge \widehat{w(v_k)} \wedge \ldots \wedge w(v_i)
\end{eqnarray*}
To show equation (\ref{eqn di map}), it is equivalent to show $P^w=P_w$. Regard $\mathbb{C}[W]$ as a natural subalgebra of $\mathbb{H}$. By using the fact that $t_wv-w(v)t_w \in \mathbb{C}[W]$ for $w \in W$,  $P^w -P_w$ is an element of the form $1 \otimes u$ for some $u \in X \otimes \wedge^i V$. Thus it suffices to show that $u=0$. To this end, by the induction hypothesis, $d_{i-1}$ is well-defined and then a direct computation (from the original expressions of $P^w$ and $P_w$) shows that $d_{i-1}(P^w-P_w)=0$ and hence $d_{i-1}(1 \otimes u)=0$. The statement now follows from the fact that the union of
\[   \left\{  e_{r_k} \otimes (x_{r_1,\ldots, r_{i-1}}) \otimes e_{r_1}\wedge \ldots \wedge \widehat{e}_{r_k} \wedge \ldots \wedge e_{r_{i-1}})\right\}_{1 \leq r_1 <  \ldots  < r_k < \ldots < r_{i-1} \leq n}
\]
and
\[   \left\{ 1 \otimes  e_{r_k}.(x_{r_1,\ldots, r_{i-1}}) \otimes e_{r_1}\wedge \ldots \wedge \widehat{e}_{r_k} \wedge \ldots \wedge e_{r_{i-1}})\right\}_{1 \leq r_1 <  \ldots  < r_k < \ldots < r_{i-1} \leq n}
\]
 forms a linearly independent set. Here $x_{r_1,\ldots, r_k} \in X$ and $e_1, \ldots, e_n$ is a fixed basis of $V$.

Verifying $d^2=0$ is straightforward.
\end{proof}

\begin{corollary} \label{cor projective resol}
\begin{enumerate}
\item For any $\mathbb{H}$-module $X$, the complex (\ref{eqn projective resolution}) forms a projective resolution for $X$.
\item The homological dimension of $\mathbb{H}$ is $\dim V$.
\end{enumerate}
\end{corollary}

\begin{proof} For (1), from Proposition \ref{prop well defined d}, we only have to show the exactness. This can be proven by an argument which imposes a grading on $\mathbb{H}$ and uses a long exact sequence (see for example \cite[Section 5.3.8]{HP}). 

We now prove (2). By (1), the homological dimension of $\mathbb{H}$ is less than or equal to $\dim V$. We now show the homological dimension attains the upper bound. Let $\gamma \in V^{\vee}$ be a regular element and let $v_{\gamma}$ be a vector with weight $\gamma \in V^{\vee}$. Define $X=\Ind_{S(V)}^{\mathbb{H}}\mathbb{C}v_{\gamma}$. By Frobenius reciprocity and using $\gamma$ is regular, $\Ext_{\mathbb{H}}^i(X,X)=\Ext_{S(V)}^i(\mathbb{C}v_{\gamma}, \mathbb{C}v_{\gamma})\neq 0$ for all $i \leq \dim V$. This shows the homological dimension has to be $\dim V$. 
\end{proof}


\subsection{Alternate form of the Koszul-type resolution}

In this section, we give another form of the differential map $d_i$, which involves the terms $\widetilde{v}$ (defined in (\ref{eqn v titlde})). There are some advantages for computations in later sections. 

We consider the maps
$\widetilde{d}_i: \mathbb{H} \otimes_{\mathbb{C}[W]} (\Res_WX \otimes \wedge^i V) \rightarrow \mathbb{H} \otimes_{\mathbb{C}[W]}(\Res_WX \otimes \wedge^{i-1}V)$ as follows:
\[  \widetilde{d}_i(h\otimes (x \otimes v_1 \wedge\ldots \wedge v_i)) =\sum_{j=0}^i (-1)^j\left( h \widetilde{v}_j \otimes x \otimes v_1 \wedge \ldots \widehat{v}_j \ldots \wedge v_i -h \otimes \widetilde{v}_j. x \otimes v_1 \wedge \ldots \widehat{v}_j \ldots \wedge v_i \right) .\]
 We show that this definition coincides with the one in the previous subsection:

\begin{proposition} \label{prop equal d}
$\widetilde{d}_i=d_i$. 
\end{proposition}

\begin{proof}
 Recall that for $v_i \in V$,
\[  \widetilde{v}_i =v_i-\sum_{\alpha \in R^+} k_{\alpha} \langle  v_i, \alpha^{\vee} \rangle t_{s_{\alpha}} .\]
Then 
\begin{eqnarray*}
& &\widetilde{v}_{r} \otimes (x \otimes v_{1} \wedge \ldots \wedge \widehat{v}_{r} \wedge \ldots \wedge v_{k}) -1 \otimes (\widetilde{v}_r.x \otimes v_1 \wedge \ldots \wedge \widehat{v}_r \wedge \ldots \wedge v_k) \\
&=&v_{r} \otimes (x \otimes v_{1} \wedge \ldots \wedge \widehat{v}_{r} \wedge \ldots \wedge v_{k}) -1 \otimes (v_r. x \otimes v_1 \wedge \ldots \wedge \widehat{v}_r \wedge \ldots \wedge v_k)\\
& & \quad -\sum_{\alpha \in R^+}k_{\alpha}\langle  v_{r}, \alpha^{\vee} \rangle  \otimes (t_{s_{\alpha}}.x) \otimes s_{\alpha}(v_{1})\wedge \ldots \wedge s_{\alpha}(\widehat{v}_{r})\wedge \ldots \wedge s_{\alpha}(v_{k}) \\
& & \quad +\sum_{\alpha \in R^+} k_{\alpha} \langle  v_{r}, \alpha^{\vee} \rangle  \otimes (t_{s_{\alpha}}.x)  \otimes v_{1}\wedge \ldots \wedge \widehat{v}_{r} \wedge \ldots \wedge v_{k} \\
&=& v_{r} \otimes (x \otimes v_{1} \wedge \ldots \wedge \widehat{v}_{r} \wedge \ldots \wedge v_{k})-1\otimes (v_{r}.x \otimes v_{1} \wedge \ldots \wedge \widehat{v}_{r} \wedge \ldots \wedge v_{k})\\
& & -(-1)^{p}\sum_{\alpha \in R^+} \sum_{p < r} k_{\alpha}\langle  v_{r}, \alpha^{\vee} \rangle \langle  v_{p}, \alpha^{\vee}  \rangle  \otimes (t_{s_{\alpha}}.x) \otimes \alpha \wedge s_{\alpha}(v_{1})\wedge \ldots  s_{\alpha}(\widehat{v}_{p})\wedge \ldots s_{\alpha}(\widehat{v}_{r})\wedge \ldots \wedge s_{\alpha}(v_{k}) \\
& & -(-1)^{p-1}\sum_{\alpha \in R^+} \sum_{ r<p} k_{\alpha}\langle  v_{r}, \alpha^{\vee} \rangle \langle  v_{p}, \alpha^{\vee}  \rangle  \otimes (t_{s_{\alpha}}.x)\otimes \alpha \wedge s_{\alpha}(v_1)\wedge \ldots  s_{\alpha}(\widehat{v}_{r})\wedge \ldots s_{\alpha}(\widehat{v}_{p})\wedge \ldots \wedge s_{\alpha}(v_{k}) \\
\end{eqnarray*}
With the expression above, some standard computations can verify $\widetilde{d}_i=d_i$. 

\end{proof}

\subsection{Euler-Poincar\'e pairing}

We define the Euler-Poincar\'e pairing as:
\[  \mathrm{EP}_{\mathbb{H}}(X, Y) = \sum_{i} (-1)^i \dim \Ext_{\mathbb{H}}^i(X,Y),
\]
where the $\Ext$ groups are defined in the category of $\mathbb{H}$-modules. This pairing can be realized as an inner product on a certain elliptic space for $\mathbb{H}$-modules analogue to the one in $p$-adic reductive groups in the sense of Schneider-Stuhler \cite{SS}. 

The elliptic pairing $\langle , \rangle_W^{\mathrm{ellip}, V}$ on $W$-representations $U$ and $U'$ is defined as
\[  \langle U, U' \rangle_W^{\mathrm{ellip}, V} =\frac{1}{|W|} \sum_{w \in W} \tr_U(w)\overline{\tr_{U'}(w)}\mathrm{det}_V(1-w) .
\]

\begin{proposition} \label{thm euler poincare}
For any finite-dimensional $\mathbb{H}$-modules $X$ and $Y$, 
\[  \mathrm{EP}_{\mathbb{H}}(X, Y) = \langle \Res_W(X), \Res_W(Y) \rangle_W^{\mathrm{ellip}, V} .\]
In particular, the Euler-Poincare pairing depends only on the $W$-module structure of $X$ and $Y$.
\end{proposition}
\begin{proof}
\begin{eqnarray*}
\mathrm{EP}_{\mathbb{H}}(X, Y) &=& \sum_{i} (-1)^i \dim \Ext_{\mathbb{H}}^i(X,Y)  \\
                               &=& \sum_i (-1)^i (\ker d_i^*-\im d_{i-1}^*)   \\
                               &=& \sum_{i} (-1)^i \dim \Hom_{\mathbb{H}}(\mathbb{H} \otimes_{\mathbb{C}[W]}(\Res_W(X)\otimes \wedge^iV), Y) \quad (\mbox{by Corollary \ref{cor projective resol}}) \\
															&=& \sum_{i} (-1)^i \dim \Hom_{\mathbb{C}[W^-]} (\Res_W(X)\otimes \wedge^iV, \Res_W(Y)) \quad \mbox{(by Frobenius reciprocity)} \\
															&=&  \sum_{w \in W} \mathrm{tr}_{\Res_WX}(w) \overline{\mathrm{tr}_{\Res_WY} (w) }  \mathrm{tr}_{\wedge^{\pm} V}(w)    \\
															&=& \langle \Res_W(X), \Res_W(Y) \rangle_{W}^{\mathrm{ellip}, V}  
\end{eqnarray*}
Here $\wedge^{\pm} V=\bigoplus_{i \in \mathbb{Z}} (-1)^i \wedge^i V$ as a virtual representation. The last equality follows from $\mathrm{tr}_{\wedge^i V}(w)=\det(1-w)$ and the definition.
\end{proof}

\section{Twisted Euler-Poincar\'e pairing}\label{s twisted EP}





Recall that $\theta$ is defined in Section \ref{ss theta action}. For any $\mathbb{H} \rtimes \langle \theta \rangle$-module $X$, denote $\Res_W X$ to be the restriction of $X$ to a $\mathbb{C}[W]$-algebra module (Definition \ref{def graded affine} (1)). The notion $\Res_{W \rtimes \langle \theta \rangle}$ is similarly defined.

\subsection{$\theta$-twisted Euler-Poincar\'e pairing} \label{ss theta twisted}

Let $X$ and $Y$ be $\mathbb{H} \rtimes \langle \theta \rangle$-modules. The differential map $d_{i}$ induces a map from $\Hom_{\mathbb{H}}(\mathbb{H}\otimes_{\mathbb{C}[W]} (\Res_WX \otimes \wedge^i V), Y)$ to $\Hom_{\mathbb{H}}(\mathbb{H}\otimes_{\mathbb{C}[W]} (\Res_WX \otimes \wedge^{i+1} V), Y)$. Then by the Frobenius reciprocity, the differential map also induces a map, denoted $d^*$ from $\Hom_{\mathbb{C}[W]}(\Res_WX \otimes \wedge^i V, \Res_WY)$ to $\Hom_{\mathbb{C}[W]}(\Res_WX \otimes \wedge^{i+1} V, \Res_WY)$ as follows:
\begin{eqnarray} \label{eqn di* actin 1}
& &  d_{i+1}^*(\psi)(x \otimes v_1 \wedge \ldots \wedge v_{i+1}) \\ 
\label{eqn di* actin 2} &=& \sum_{j=0}^{i+1} (-1)^j v_j.\psi( x \otimes v_1 \wedge \ldots \widehat{v}_j \ldots \wedge v_{i+1})-\sum_{j=0}^i (-1)^j\psi( v_j.x \otimes v_1 \wedge \ldots \widehat{v}_j \ldots \wedge v_{i+1}),
			\end{eqnarray}


Define $\theta^*$ to be the linear automorphism on $\Hom_{\mathbb{C}[W]}(\Res_WX \otimes \wedge^i V, \Res_WY)$ given by
\begin{eqnarray} \label{eqn theta* actin 1} \theta^*(\psi)(x\otimes v_1\wedge \ldots \wedge v_i)&=&\theta\circ \psi(\theta(x)\otimes \theta(v_1)\wedge\ldots \wedge \theta(v_i)) .
\end{eqnarray}
Here $\theta$-actions on $\Res_W X$ and $\Res_W Y$ are just the natural actions from the $\theta$-actions on $X$ and $Y$ (as $\mathbb{H} \rtimes \langle \theta \rangle$-modules), and furthermore the $\theta$-action on $v_i$ comes from the action of $\theta$ on the corresponding Dynkin diagram.

\begin{lemma} \label{lem d theta commut}
$\theta^* \circ d^*= d^* \circ \theta^*$
\end{lemma}

\begin{proof}
\begin{eqnarray*}
 & &(\theta^*\circ d^*)(\psi)(x\otimes v_1\wedge \ldots \wedge v_k)  \\
 &=& \theta \circ d^*(\psi)(\theta(x) \otimes \theta(v_1)\wedge \ldots \wedge \theta(v_k) )\\
&=&  \theta\circ \psi(d(\theta(x) \otimes \theta(v_1) \wedge \ldots \wedge \theta(v_k)))  \\
&=&  \sum_i (-1)^iv_r.\theta \circ \psi(\theta(x) \otimes \theta(v_1)  \wedge \ldots \wedge \theta(\widehat{v}_i) \wedge \ldots \wedge \theta(v_k)) \\
& & \quad -\sum_i (-1)^i \theta \circ \psi(\theta(v_r).\theta(x) \otimes \theta(v_1)  \wedge \ldots \wedge \theta(\widehat{v}_i) \wedge \ldots \wedge \theta(v_k) )\\
&=&  \sum_i (-1)^iv_r.\theta^*(\psi)(x \otimes v_1  \wedge \ldots \wedge \widehat{v}_i \wedge \ldots \wedge v_k) \\
& & \quad -\sum_i (-1)^i  \theta^*(\psi)(v_r.x \otimes v_1  \wedge \ldots \wedge \widehat{v}_i \wedge \ldots \wedge v_k) \\
&=& (d^* \circ \theta^*)(\psi)(x\otimes v_1\wedge \ldots \wedge v_k)  
\end{eqnarray*}

\end{proof}

 By Lemma \ref{lem d theta commut}, $\theta^*$ induces an action, still denoted $\theta^*$ on $\Ext^i_{\mathbb{H}}(X, X)$. We can then define the $\theta$-twisted Euler-Poincar\'e pairing $\mathrm{EP}_{\mathbb{H}}^{\theta}$ as follows:
\begin{definition}
For $\mathbb{H} \rtimes \langle \theta \rangle$-modules $X$ and $Y$, define
\[  \mathrm{EP}_{\mathbb{H}}^{\theta}(X, Y) = \sum_i (-1)^i \mathrm{trace}(\theta^*: \Ext_{\mathbb{H}}^i(X,Y) \rightarrow \Ext_{\mathbb{H}}^i(X, Y)) . \]
Here we also regard $X$ and $Y$ to be $\mathbb{H}$-modules equipped with the $\theta$-action.
\end{definition}
We remark that this definition also makes sense for $\theta$ to be any automorphism of $\mathbb{H}$. However, when we prove Theorem \ref{thm twisted ext} later, we essentially require $\theta$ to arise from $w_0$ in (\ref{eqn involution}).




\subsection{$\theta$-twisted elliptic pairing on Weyl groups}

We review the $\theta$-twisted elliptic representation theory of Weyl groups in \cite{CH}.

\begin{definition}
An element $w \in W$ is said to be {\it $\theta$-elliptic} if $\mathrm{det}_V(1-w\theta)\neq 0$. A $\theta$-twisted conjugacy class is the set $\left\{ ww_1\theta(w)^{-1} : w_1 \in W \right\}$ for some $w \in W$. A $\theta$-twisted conjugacy class is said to be {\it elliptic} if it contains an $\theta$-elliptic element.
\end{definition}
Define
\begin{equation} \label{eqn J theta}
 \mathcal J^{\theta} = \left\{  J \subsetneq \Delta :  \theta(J)=J   \right\}. 
\end{equation}

\begin{lemma} \label{lem not elliptic}
\begin{enumerate}
\item If $w \in W$ is not a $\theta$-elliptic element, then $w$ is $\theta$-conjugate to an element in $W_J$ for some $J \in \mathcal J^{\theta}$.
\item Let $J \in \mathcal J^{\theta}$. If $w \in W_J$, then there exists a non-zero $\gamma \in V$ such that $w\theta(\gamma)=\gamma$.
\end{enumerate}
\end{lemma}
\begin{proof}
We first prove (1). Suppose $w$ is not $\theta$-elliptic element. Then there exists $\gamma \in V$ such that $w\theta(\gamma)=\gamma$. We may choose $w_1 \in W$ such that $w_1(\gamma)$ lies in the fundamental chamber. Let $\gamma'=w_1(\gamma)$.  Then the stabilizer for $\gamma'$ is $W_J$ for some $J \subset \Delta$. Since $\gamma'$ is in the fundamental chamber, $\theta(\gamma')$ is also in the fundamental chamber. The fact that $w_1w\theta(w_1^{-1})\theta(\gamma')=\gamma'$ with standard theory for root systems (see for example \cite[Theorem 1.12(a)]{Hu}) forces $\theta(\gamma')=\gamma'$. We have $w_1w\theta(w_1^{-1}) \in W_J$. It remains to show that $J \in \mathcal J^{\theta}$. For $w$ with $w(\gamma')=\gamma'$, we also have $\theta(w)(\gamma')=\theta(w)\theta(\gamma')=\theta(w(\gamma'))=\gamma'$. Hence $\theta(J)=J$ and so $J \in \mathcal J^{\theta}$ as desired.

For (2), choose $\gamma \in V_J^{\bot}$. Then $\theta(\gamma)=\gamma$ and so $w\theta(\gamma)=\gamma$ for any $\gamma \in W_J$.

\end{proof}
\begin{definition} \cite{CH} \label{def twisted ell pairing}
For any $W \rtimes \langle \theta \rangle$-representation $U$ and $U'$, the $\theta$-twisted elliptic pairing on $U$ and $U'$ is defined as:
\[   \langle U, U' \rangle^{\theta-\mathrm{ellip}, V}_W=\frac{1}{|W|}\sum_{w \in W} \tr_U(w\theta)\overline{\tr_{U'}(w\theta)}\mathrm{det}_V(1-w\theta) .
\]
Since $w_0\theta=-\mathrm{Id}_V$ on $V$, it is equivalent that
\begin{align*}
  \langle U, U' \rangle^{\theta-\mathrm{ellip}, V}_{W} &= \frac{1}{|W|}\sum_{w \in W} \tr_{U^+-U^-}(ww_0) \overline{\tr_{U'^+-U'^-}(ww_0)}\mathrm{det}_V(1+ww_0), 
\end{align*}
where $U^+$ and $U^-$ (resp. $U'^+$ and $U'^-$) are the $+1$ and $-1$-eigenspaces of $w_0\theta$ of $U$ (resp. $U'$), and $U^+-U^-$ and $U'^+-U'^-$ are regarded as virtual representations of $W$.
\end{definition}

Let $R(W \rtimes \langle \theta \rangle)$ be the virtual representation ring of $W \rtimes \langle \theta \rangle$. Since $\theta$ is an inner automorphism on $W$, $\Res_W U$ is an irreducible $W$-representation for any irreducible $W \rtimes \langle \theta \rangle$ representation $U$. Then there exists a unique $W \rtimes \langle \theta \rangle$ representation denoted $\overline{U}$ such that $U$ and $\overline{U}$ are isomorphic as $W$-representation but non-isomorphic as $W \rtimes \langle \theta \rangle$-representation. Let $R'$ be the space spanned by $U \oplus  \overline{U}$ for all $U \in \Irr(W \rtimes \langle \theta \rangle)$. 
Let
\[   \overline{R}_W = R(W \rtimes \langle \theta \rangle) / R'   .\]
Note that $\overline{R}_W$ is isomorphic to $R(W)$ as vector spaces, but there is no canonical isomorphism between them. Note that $R'$ is in the radical of $\langle , \rangle^{\theta-\mathrm{ellip},V}_W$ and so $\langle , \rangle^{\theta-\mathrm{ellip},V}_W$ descends to $ \overline{R}_W$. A natural question is to describe $\rad \langle , \rangle^{\theta-\mathrm{ellip},V}_W$ and is answered in Proposition \ref{prop radical desc}.


\begin{lemma} \label{lem fr trick}
Let $U \in R(W \rtimes \langle \theta \rangle)$. Let $J \in \mathcal J^{\theta}$ and let $U' \in R(W_J \rtimes \langle \theta \rangle)$. If 
\[ \sum_{w \in W} \tr_U(w\theta)\overline{\tr_{\Ind_{W_J}^WU'}(w\theta) } =0, \] then
\[     \sum_{w \in W_J} \tr_U(w\theta)\overline{\tr_{U'}(w\theta)} =0 .
\]
\end{lemma}

\begin{proof}
This follows from the following: 
\begin{align*}
0=&   \sum_{w \in W} \tr_U(w\theta)\overline{\tr_{\Ind_{W_J \rtimes \langle \theta \rangle}^{W \rtimes \langle \theta \rangle}U'}(w\theta) }  \\
=&   2|W|\langle U, \Ind_{W_J \rtimes \langle \theta\rangle}^{W \rtimes \langle \theta \rangle} U'\rangle_{W \rtimes \langle \theta \rangle}-|W|\langle U, \Ind^W_{W_J}U' \rangle_{W} \\
=&   2|W|\langle \Res_{W_J\rtimes \langle \theta \rangle}U, U' \rangle_{W_J \rtimes \langle \theta \rangle}-|W|\langle \Res_{W_J}U, U' \rangle_{W_J} \\
=&   \frac{|W|}{|W_J|}\sum_{w \in W_J} \tr_U(w\theta)\overline{\tr_{U'}(w\theta) }
\end{align*} 
Here $\langle , \rangle_W$ and $\langle , \rangle_{W_J}$ denotes the standard inner form on $W$-representations and $W_J$-representations respectively.
\end{proof}





\begin{proposition} \label{prop radical desc}
\begin{enumerate}
\item The radical of $\langle , \rangle^{\theta-\mathrm{ellip},V}_W$ on $\overline{R}_W$ is the image of 
\[\bigoplus_{J \in \mathcal J^{\theta}} \Ind_{W_J\rtimes \langle \theta \rangle}^{W \rtimes \langle \theta \rangle} R(W_J \rtimes \langle \theta \rangle). \]
\item The dimension of the quotient space $\overline{R}_W/\mathrm{rad}\langle , \rangle^{\theta-\mathrm{ellip},V}_W$ is equal to the number of  elliptic $\theta$-twisted conjugacy classes.
\end{enumerate}
\end{proposition}

\begin{proof}
We first prove (1). The proof follows the one in \cite[Proposition 2.2.2]{Re}.
Let $U \in \Ind_{W_J \rtimes \langle \theta \rangle}^{W \rtimes \langle \theta \rangle} R(W_J \rtimes \langle \theta \rangle)$ for some $J \in \mathcal L^{\theta}$. Then $\chi_U(w\theta)$ vanishes for all $w$ that is not $\theta$-twisted conjugate to an element in $W_J$. Then by Lemma \ref{lem not elliptic} (2), $\bigoplus_{J \in \mathcal L^{\theta}} \Ind_{W_J \rtimes \langle \theta \rangle}^{W \rtimes \langle \theta \rangle} \mathcal R(W_J \rtimes \langle \theta \rangle)$ is a subset of the radical of $\langle , \rangle^{\theta-\mathrm{ellip},V}_W$. 

We now prove the converse direction. We pick a virtual representation $U \in  \mathrm{rad}\langle , \rangle^{\theta-\mathrm{ellip},V}_W$ such that $\langle U, \Ind_{W_J}^W U'  \rangle_{W \rtimes \langle \theta \rangle}=0$ for all $J \in \mathcal L^{\theta}$ and $U' \in R(W_J \rtimes \langle \theta \rangle)$. By Lemma \ref{lem fr trick}, $\mathrm{tr}_U(w\theta)=0$ for all $w \in W_J$ and all $J \in \mathcal J^{\theta}$. By Lemma \ref{lem not elliptic}, $\mathrm{tr}_U(w\theta)=0$ for any non-elliptic element $w$. This implies that $\tr_U(w\theta)=\tr_{U^+-U^-}(ww_0)=0$ for all $w$, where $U^+$ and $U^-$ are the $+1$ and $-1$ eigenspaces for $w_0\theta$. 
Hence $U^+=U^-$ and by definition $U \in \overline{R}_W$. Thus the orthogonal complement of the image of $\bigoplus_{J \in \mathcal L^{\theta}} \Ind_{W_J}^W R(W_J\rtimes \langle \theta \rangle)$ in $\overline{R}_W$ with respect to the pairing $ \rad\langle , \rangle^{\theta-\mathrm{ellip},V}_W$ is exactly zero. This proves (1).

For (2), it follows from Definition \ref{def twisted ell pairing} and the fact that $\mathrm{det}_V(1-w\theta)$ is non-zero if and only if $w$ is $\theta$-elliptic.


\end{proof}

\subsection{Relation between two twisted elliptic pairings}

\begin{notation}
Let $X$ be an $\mathbb{H} \rtimes \langle \theta \rangle$-module.  Define $X^{\pm}$ to be the $\pm 1$ eigenspaces of the action of $\theta t_{w_0}$ on $X$ respectively. It is easy to see  $X^{\pm}$ are invariant under the action of $t_w$ for $w \in W$ (see Lemma \ref{lem v tilde action} below). We shall regard  $X^{\pm}$ as $W$-representations or $W \rtimes \langle \theta \rangle$-representations. Moreover, since $\theta t_{w_0}$ is diagonalizable, we also have $X=X^+ \oplus X^-$.
\end{notation}

\begin{lemma} \label{lem v tilde action}
Let $X$ be an $\mathbb{H} \rtimes \langle \theta \rangle$-module. Then
\begin{enumerate}
\item $X^+$ and $X^-$ are $W \rtimes \langle \theta \rangle$-invariant
\item Let $X$ be an $\mathbb{H} \rtimes \langle \theta \rangle$-module. For any $v \in V$, $\widetilde{v}.X^{\pm} \subset X^{\mp}$.
\end{enumerate}
\end{lemma}

\begin{proof}
(1) follows from $\theta t_{w_{0}}t_w=t_wt_{w_0}\theta$. (2) follows from $w_0\theta(v)=-v$ and Lemma \ref{lem tilde element}. 

\end{proof}

\begin{lemma} \label{lem v tilde action c}
For $\mathbb{H} \rtimes \langle \theta \rangle$-modules $X$ and $Y$, define 
\[\Hom^+_i = \Hom_{\mathbb{C}[W]}(X^+\otimes \wedge^i V, Y^+) \oplus \Hom_{\mathbb{C}[W]}(X^-\otimes \wedge^{i}V, Y^-) \]
 and 
\[ \Hom^-_i=\Hom_{\mathbb{C}[W]}(X^+\otimes \wedge^iV,Y^-)\oplus \Hom_{\mathbb{C}[W]}(X^-\otimes \wedge^iV, Y^+). \]
The map $d_i^*$ sends $\Hom^{\pm}_i \rightarrow \Hom^{\mp}_{i+1}$. Moreover, $\theta^*$ acts identically as $(-1)^i$ on $\Hom_i^+$ and acts identically as $-(-1)^i$ on $\Hom_i^-$. 
\end{lemma}

\begin{proof}

The first assertion follows from Lemma \ref{lem v tilde action} and Proposition \ref{prop equal d}. For the second assertion, we pick $\psi \in \Hom^+_i$. Suppose $x \in X^+$ and $v_1, \ldots, v_i \in V$. Then
\begin{align*}
 & \theta^*(\psi)(x \otimes v_1 \wedge \ldots \wedge v_i)  \\
=& \theta.\psi(\theta(x) \otimes \theta(v_1)\wedge \ldots \wedge \theta(v_i)) \\
=& t_{w_0}\theta.\psi((t_{w_0}\theta.x) \otimes w_0\theta(v_1) \wedge \ldots \wedge w_0\theta(v_i)) \\
=& (-1)^it_{w_0}\theta.\psi(x \otimes v_1 \wedge \ldots \wedge v_i) \\
=& (-1)^i\psi(x \otimes v_1 \wedge \ldots \wedge v_i)
\end{align*}
The forth equality follows from $w_0\theta(v)=-v$, $t_{w_0}\theta.x=x$, and the last equality follows from $\im\psi \in Y^+$. Other cases are similar.
\end{proof}

With $\Hom_i^{\pm}$ defined in Lemma \ref{lem v tilde action c}, we also define that 
\[   \Ext^i(X, Y)^+ = \frac{\ker(d_i^*: \Hom_i^+ \rightarrow \Hom_i^-)}{\im(d_i^*: \Hom_i^- \rightarrow \Hom_i^+)}, \]
and similarly, 
\[   \Ext^i(X, Y)^- = \frac{\ker(d_i^*: \Hom_i^- \rightarrow \Hom_i^+)}{\im(d_i^*: \Hom_i^+ \rightarrow \Hom_i^-)} .\]
Note that by the projective resolution in (\ref{eqn projective resolution}),
 \begin{align} \label{eqn decomp ext}
 \Ext_{\mathbb{H}}^i(X, Y) =\Ext^i(X, Y)^+ \oplus \Ext^i(X, Y)^-. 
\end{align}

\begin{theorem} \label{thm twisted ext}
For any finite-dimensional $\mathbb{H} \rtimes \langle \theta \rangle$-modules $X$ and $Y$ with $\theta$ defined as in (\ref{eqn involution}), 
\[    \mathrm{EP}^{\theta}_{\mathbb{H}}(X, Y)=\langle \Res_{W\rtimes \langle \theta\rangle}X, \Res_{W \rtimes \langle \theta \rangle}Y \rangle^{\theta-\mathrm{ellip}, V}_{W} .
\]
In particular, the $\theta$-twisted elliptic pairing $\mathrm{EP}^{\theta}_{\mathbb{H}}$ depends on the $W$-module structures of $X$ and $Y$ only.
\end{theorem}

\begin{proof}
Set $d_i^{*,+}=d_i^*|_{\Hom_i^+}$ and $d_i^{*, -}=d_i^*|_{\Hom_i^-}$.
\begin{eqnarray*}
& & \mathrm{EP}_{\mathbb{H}}^{\theta}(X, Y) \\
&=& \sum_{i} (-1)^i \mathrm{trace}(\theta^*: \Ext^i_{\mathbb{H}}(X, Y) \rightarrow \Ext^i_{\mathbb{H}}(X, Y)) \\
&=& \sum_{i} (-1)^i [(-1)^i\dim \Ext^i(X, Y)^+ -(-1)^i \dim \Ext^i(X, Y)^-] \quad \mbox{ (by (\ref{eqn decomp ext}) and Lemma \ref{lem v tilde action c})}\\
&=&  \sum_{i} (\dim \Ext^i(X, Y)^+- \dim \Ext^i(X, Y)^-) \\
&=&  \sum_i [(\dim \ker d_i^{*,+}-\dim \im d_{i-1}^{*,-})-(\dim \ker d_i^{*,-}-\dim \im d_{i-1}^{*,+})] \\
&=& \sum_i (\dim \ker d_i^{*,+} +\dim \im d_{i-1}^{*,+}))-(\dim \ker d_i^{*,-}+\dim \im d_{i-1}^{*,-}) \\
&=&  \sum_i (\dim \Hom^+_i- \dim \Hom^-_i) \quad \mbox{ (definition of $\Hom^{\pm}$ in Lemma \ref{lem v tilde action c})}\\
&=& \frac{1}{|W|} \sum_{w \in W} \tr_{X^+-X^-}(w)\overline{\tr_{Y^+-Y^-}(w)}\mathrm{det}_V(1+w) \quad (\mbox{as virtual representations}) \\
&=&  \frac{1}{|W|}  \sum_{w \in W} \tr_{X}(ww_0\theta)\overline{\tr_{Y}(ww_0\theta)}\mathrm{det}_V(1-ww_0\theta) \\
&=& \langle \Res_{W\rtimes \langle \theta \rangle} (X), \Res_{W \rtimes \langle \theta \rangle} (Y) \rangle^{\theta-\mathrm{ellip}, V}_W 
\end{eqnarray*}
The third last equality follows from the fact that $\sum_i \tr_{\wedge^iV}(w)=\mathrm{det}_V(1+w)$ and $w_0\theta=-\mathrm{Id}_V$.
\end{proof}

\begin{remark}
We give an example to show that Theorem \ref{thm twisted ext} is not true in general if $\theta$ is replaced by an outer automorphism on $W$. Let $R$ be of type $A_1 \times A_1$. Let $\theta'$ be the Dynkin diagram automorphism interchanging two factors of $A_1$. Let $\mathbb{H}$ be the graded Hecke algebra of type $A_1 \times A_1$. Note that $\langle , \rangle^{\theta'-\mathrm{ellip}, V}_W \equiv 0$ as $\mathrm{tr}(w\theta')=0$ for all $w \in W$. Here $W=S_2 \times S_2$ and $V=\mathbb{C} \oplus \mathbb{C}$. However, we may choose an $\mathbb{H}$-module $X$ (e.g. the exterior tensor product of Steinberg modules) such that $\mathrm{EP}^{\theta'}_{\mathbb{H}}(X, X) \neq 0$.



\end{remark}

We give an interpretation of $\theta$-twisted Euler-Poincar\'e pairing with the Euler-Poincar\'e pairing of $\mathbb{H} \rtimes \langle \theta \rangle$-modules. Define $\mathrm{EP}_{\mathbb{H} \rtimes \langle \theta \rangle}(X, Y)=\sum_{i} (-1)^i \dim\Ext^i_{\mathbb{H} \rtimes \langle \theta \rangle}(X, Y)$, where $\Ext^i_{\mathbb{H} \rtimes \langle \theta \rangle}$ is taken in the category of $\mathbb{H} \rtimes \langle \theta \rangle$-modules.
\begin{corollary} \label{cor extended}
For any finite-dimensional $\mathbb{H} \rtimes \langle \theta \rangle$-modules $X$ and $Y$,
\[  \dim \mathrm{Ext}_{\mathbb{H} \rtimes \langle \theta \rangle}^i(X, Y) =\frac{1}{2} \dim \mathrm{Ext}_{\mathbb{H}}^i(X, Y)+\frac{1}{2}\mathrm{trace}(\theta^*: \mathrm{Ext}^i_{\mathbb{H}}(X, Y) \rightarrow \mathrm{Ext}^i_{\mathbb{H}}(X, Y) ) , \]
and
\[ \mathrm{EP}_{\mathbb{H}\rtimes \langle \theta \rangle}(X,Y) = \frac{1}{2}\mathrm{EP}_{\mathbb{H}}(X,Y)+\frac{1}{2}\mathrm{EP}^{\theta}_{\mathbb{H}}(X,Y) .\]
\end{corollary}

\begin{proof}
Note that 
\[   \Hom_{\mathbb{C}[W] \rtimes \langle \theta \rangle} (\mathrm{Res}_{W \rtimes \langle \theta \rangle}X \otimes \wedge^i V, \mathrm{Res}_{W \rtimes \langle \theta \rangle} Y) \cong\left\{ \begin{array}{rl} \mathrm{Hom}^+_i &  \mbox{ if $i$ is even } \\ \mathrm{Hom}^-_i  & \mbox{ if $i$ is odd }  \end{array} \right.
\]
Then by using a Koszul type resolution as in (\ref{eqn projective resolution}), one could see that
\[   \mathrm{Ext}^i_{\mathbb{H} \rtimes \langle \theta \rangle}(X, Y) = \left\{ \begin{array}{rl} \mathrm{Ext}^+_i &  \mbox{ if $i$ is even } \\ \mathrm{Ext}^-_i  & \mbox{ if $i$ is odd }  \end{array} \right.
\]
By Lemma \ref{lem v tilde action c}, the latter expression above is equal to 
\[ \frac{1}{2}\dim \mathrm{Ext}_{\mathbb{H}}^i(X, Y)+\frac{1}{2}\mathrm{trace}(\theta^*: \mathrm{Ext}^i_{\mathbb{H}}(X, Y) \rightarrow \mathrm{Ext}^i_{\mathbb{H}}(X, Y) ) .\]

It follows from the proof of Proposition \ref{thm euler poincare} that
\begin{align*}  
& \Ext_{\mathbb{H}\rtimes \langle \theta \rangle}(X,Y) \\
=& \frac{1}{2|W|}\sum_{w \in W } \tr_X(w)\overline{\tr_Y(w)}\mathrm{det}_V(1-w)+\frac{1}{2|W|}\sum_{w \in W}\tr_X(w\theta)\overline{\tr_Y(w\theta)}\mathrm{det}_V(1-w\theta)  \\
=&\frac{1}{2} \langle \Res_{W}(X), \Res_W(Y) \rangle^{\mathrm{ellip}, V}_W+\frac{1}{2}\langle \Res_{W \rtimes \langle \theta \rangle}(X), \Res_{W\rtimes \langle \theta \rangle}(Y) \rangle^{\theta-\mathrm{ellip}, V}_W 
\end{align*}
Now the statement follows from Theorem \ref{thm twisted ext} and Proposition \ref{thm euler poincare}.
\end{proof}

\begin{corollary} \label{cor relation radical}
Let $X$ be a finite-dimensional $\mathbb{H} \rtimes \langle \theta \rangle$-module. If $X \in \mathrm{rad}(\mathrm{EP}_{\mathbb{H}}^{\theta})$, then $X \in \mathrm{rad}(\mathrm{EP}_{\mathbb{H}})$.
\end{corollary}

\begin{proof}
Proposition \ref{prop radical desc} is still valid if we replace $\theta$ by $\Id$ and replace $\mathcal J^{\theta}$ by $\mathcal J$, where $\mathcal J$ is the set of all proper subsets of $\Delta$. Since $\mathcal J^{\theta} \subseteq \mathcal J$, the statement follows from Proposition \ref{prop radical desc} and Theorem \ref{thm twisted ext}.

\end{proof}

\subsection{Semi-positiveness of the twisted Euler-Poincar\'e pairing}
Let $\widetilde{W}$ be the spin cover of $W$. For $\dim V$ even, let $S$ be the irreducible basic spin representations of $\widetilde{W}$. For $\dim V$ odd, let $S^+$ and $S^-$ be the two distinct basic spin representations of $\widetilde{W}$-representation and let $S=S^+ \oplus S^-$. For a more detail discussion of the spin cover $\widetilde{W}$ or the representation $S$, one may refer to \cite{BCT}, \cite{Ch1} or \cite{CT}. The only property we will use in this paper is the following:
\[  S \otimes S = n \wedge^{\bullet} V ,\]
where $n=1$ when $\dim V$ is even and $n=2$ when $\dim V$ is odd. For an $\mathbb{H} \rtimes \langle \theta \rangle$-module $X$, we define $\theta$-twisted Dirac index as:
\[  I^{\theta}(X)=(X^+-X^-) \otimes S , \]
as a virtual $\widetilde{W}$-representation. The terminology of the $\theta$-twisted Dirac index comes from the form of the Dirac index defined by Ciubotaru-Trapa \cite{CT} and Ciubotaru-He \cite{CH}. 
\begin{proposition} \label{prop ep dirac index}
For $\mathbb{H} \rtimes \langle \theta \rangle$-modules $X_1$ and $X_2$, 
\[   \frac{n}{2}\mathrm{EP}_{\mathbb{H}}^{\theta}(X_1, X_2)=\langle I^{\theta}(X_1), I^{\theta}(X_2) \rangle_{\widetilde{W}} , \]
where $n=1$ if $\dim V$ is even and $n=2$ if $\dim V$ is odd. Here $\langle , \rangle_{\widetilde{W}}$ is the standard inner product on $\widetilde{W}$-representations.
\end{proposition}

\begin{proof}
The proof is similar to the one in \cite[Proposition 3.1]{CT}.
\begin{eqnarray*}
& &\langle I^{\theta}(X_1), I^{\theta}(X_2) \rangle_{\widetilde{W}} \\
&=& \langle (X^+_1-X^-_1) \otimes S, (X^+_2-X^-_2)\otimes S \rangle_{\widetilde{W}}  \\
 &=& \langle X_1^+-X_1^-, (X^+_2-X^-_2) \otimes S \otimes S \rangle_{\widetilde{W}} \\
&=&  n\langle X_1^+-X_1^-, (X_2^+-X_2^-) \otimes \wedge^{\bullet} V \rangle_{\widetilde{W}} \\
&=&  \frac{n}{2}\langle X_1, X_2 \rangle^{\mathrm{\theta-ellip,} V}_W \\
&=& \frac{n}{2} \mathrm{EP}_{\mathbb{H}}^{\theta}(X_1, X_2) \quad \mbox{(by Theorem \ref{thm twisted ext}) }
\end{eqnarray*}
\end{proof}

\begin{corollary}  \label{cor semipositive ep}
The $\theta$-twisted Euler-Poincar\'e pairing $\mathrm{EP}^{\theta}_{\mathbb{H}}$ is semi-positive definite.
\end{corollary}


\subsection{Twisted elliptic space}
Let $K_{\mathbb{C}}(\mathbb{H} \rtimes \langle \theta \rangle)$ be the Grothendieck group of the category of finite-dimensional $\mathbb{H}$-modules over $\mathbb{C}$. We have seen from Theorem \ref{thm twisted ext} that $\mathbb{EP}^{\theta}_{\mathbb{H}}$ does not depend on the choice of a representative of an element in $K_{\mathbb{C}}(\mathrm{Mod}_{\mathrm{fin}}(\mathbb{H} \rtimes \langle \theta \rangle))$. Hence we can extend $\mathbb{EP}^{\theta}_{\mathbb{H}}$ to a Hermitian form, still denoted $\mathrm{EP}^{\theta}_{\mathbb{H}}$ on  $K_{\mathbb{C}}(\mathrm{Mod}_{\mathrm{fin}}(\mathbb{H} \rtimes \langle \theta \rangle))$.

 For any irreducible $\mathbb{H} \rtimes \langle \theta \rangle$-module $X$, there are two possibilities:
\begin{enumerate}
\item Suppose $X|_{\mathbb{H}}$ is reducible. Then $X|_{\mathbb{H}}$ is the sum of two non-isomorphic irreducible $\mathbb{H}$-modules, denoted $X_1$ and $X_2$. In this case, $\theta(X_1)=X_2$ and so $\tr_{\Res_{W}X}(w\theta)=0$ for all $w \in W$. By Theorem \ref{thm twisted ext}, $X$ is in $\rad(\mathrm{EP}_{\mathbb{H}}^{\theta})$. 
\item Suppose $X|_{\mathbb{H}}$ is irreducible. Then there exists another $\mathbb{H} \rtimes \langle \theta \rangle$-module, denoted $\overline{X}$ such that $\overline{X}$ and $X$ are isomorphic as $\mathbb{H}$-modules, but non-isomorphic as $\mathbb{H} \rtimes \langle \theta \rangle$-modules. More precisely, let $\pi_X$ and $\pi_{\overline{X}}$ be the maps defining the action of $\mathbb{H} \rtimes \langle \theta \rangle$ on $X$ and $\overline{X}$ respectively. Those maps satisfy $\pi_{\overline{X}}(\theta)=-\pi_X(\theta)$. This implies $X \oplus \overline{X}$ lies in $\rad(\mathrm{EP}_{\mathbb{H}}^{\theta})$ by Theorem \ref{thm twisted ext}. 
\end{enumerate}
Let $K^1$ be the subspace of $\rad(\mathrm{EP}_{\mathbb{H}}^{\theta})$ spanned by all $X$ with $X$ in case (1) (i.e. $X|_{\mathbb{H}}$ being reducible). Let $K_2$ be the subspace of $\rad(\mathrm{EP}_{\mathbb{H}}^{\theta})$ spanned by all $X \oplus \overline{X}$ for all $X$ in case (2) (i.e. $X|_{\mathbb{H}}$ being reducible). 
We define the space 
\[ K_{\mathbb{H}}^{\theta}  =  K_0(\mathrm{Mod}_{\mathrm{fin}}(\mathbb{H} \rtimes \langle \theta \rangle)) /(K^1 \oplus K^2) . \]
Note that the image of all irreducible $\mathbb{H}$-modules $X$ with the property that $X^{\theta} \cong X$ forms a basis on $K_{\mathbb{H}}^{\theta}$. 

Since $K^1$ and $K^2$ are in the radical of $\mathrm{EP}^{\theta}_{\mathbb{H}}$, $\mathrm{EP}^{\theta}_{\mathbb{H}}$ descends to $K_{\mathbb{H}}^{\theta}$. We define the twisted elliptic space to be:
\[   \mathrm{Ell}_{\mathbb{H}}^{\theta}=K_{\mathbb{H}}^{\theta}/ \mathrm{rad}(EP^{\theta}_{\mathbb{H}}) . \]

\begin{corollary}
The space $\mathrm{Ell}_{\mathbb{H}}^{\theta}$ is equipped with $\mathrm{EP}^{\theta}_{\mathbb{H}}$ as an inner product. 
\end{corollary}

\begin{proof}
The assertion follows from Corollary \ref{cor semipositive ep} and our construction of $\mathrm{Ell}_{\mathbb{H}}^{\theta}$. 
\end{proof}

The space $\mathrm{rad}(EP^{\theta}_{\mathbb{H}})$ will be discussed more in Section \ref{s solvable induced}.



\section{$\theta^*$-action on $\mathrm{Ext}$-groups of rigid modules} \label{s theta action}

\subsection{ $\Ext$-groups of rigid modules}

Recall that tempered modules are defined in Definition \ref{def tempered}. The notion for a parabolically induced module is given in Notation \ref{not induced modules}.



The rigid modules are parabolically induced and tempered modules with a special kind of induced data described in the following definition. 
\begin{definition} \label{def rigid}
 
Let $\mathcal J_{\mathrm{rig}}$ be the collection of subsets $J$ of $\Delta$ such that 
\begin{eqnarray} \label{eqn card w J}
  \mathrm{card}(\left\{ w \in W: w(J) =J \right\}) =1 .
	\end{eqnarray}
 Let $\Xi_{\mathrm{rig}}$ be the collection of $(J, U, \nu) \in \Xi$ such that $J \in \mathcal J_{\mathrm{rig}}$. An $\mathbb{H}$-module $X$ is said to be a {\it rigid module} if $X=X(J, U)$ for some $(J, U, 0) \in \Xi_{\mathrm{rig}}$. In particular, a rigid module is a tempered and parabolically induced module.


\end{definition}


\begin{remark} \label{rmk degenerate limits}
We give two remarks on our definition of rigid modules:
\begin{enumerate}
\item[(1)] The term ''rigid'' refers to the special choice of $J$ in the induction datum for a rigid module. Such induction datum provides nice structures such as discussed Lemma \ref{lem structure rigid} and Lemma \ref{lem rigid irr} below for computing the $\Ext$-groups and $\theta^*$-action without introducing more tools.
\item[(2)] The essential algebraic structure we need in our later computations is descried in Lemma \ref{lem structure rigid}. The way we formulate the definition is easier to connect to the tempered modules in Section \ref{s solvable induced}. As mentioned in the introduction, rigid modules provide examples of solvable tempered modules, which will be discussed in the Section \ref{s solvable induced}.

\end{enumerate}
\end{remark}






\begin{remark} \label{rmk classify rigid}
For the case $\theta=\Id_V$ (i.e. non-simply laced types, $E_7$, $E_8$ and $D_n$ ($n$ even)), $w_0w_{J}(J)=J$ for any $J$ and hence only $\Delta$ can satisfy (\ref{eqn card w J}). For the case that $\theta \neq \Id_V$ (i.e $A_n$, $D_n$ ($n$ odd) and $E_6$), $J \subset \Delta$ satisfies (\ref{eqn card w J}) in Definition \ref{def rigid} if and only if $J=\Delta$ or $J$ is in one of the following case:
\begin{enumerate}
\item in type $A_n$ and if we identify subsets of $\Delta$ (up to conjugation in $W$) with partitions of $n$, $J$ corresponds to a partition of distinct parts, or equivalently $J$ is of type $A_{m_1} \times \ldots A_{m_k}$ with all $m_i$ mutually distinct and $m_1+\ldots+m_k=n-k$ or $n-k+1$;
\item $D_{n}$ ($n$ odd) and $J$ is of type $A_{n-1}$;
\item $E_6$ and $J$ is of type $D_5$ or $A_4 \times A_1$.
\end{enumerate}

From the classification, it is easy to see that all rigid modules satisfy (1) in the three conditions of Proposition \ref{prop herm}.

\end{remark}



\begin{lemma} \label{lem theta unstable}
Let $J$ be a subset of $\Delta$. If $J \in \mathcal J_{\mathrm{rig}}$, then there does not exist $J' \in \mathcal J^{\theta}$ such that $w(J) \subset J' \subsetneq \Delta$ for some $w \in W$. Here $\mathcal J_{\mathrm{rig}}$ is defined in (\ref{eqn J theta}).
\end{lemma}

\begin{proof}
This is an easy case-by-case checking with the use of Remark \ref{rmk classify rigid}.

\end{proof}

To analyze the structure of rigid modules, we need the following result in \cite{BM} about weight spaces:

\begin{proposition} \cite{BM} \label{prop weights}
Let $(J, U, \nu) \in \Xi$ and $X=X(J, U, \nu)$. Then the weights of the $\mathbb{H}$-module $X$ are
\begin{eqnarray} \label{eqn weight set}  \left\{   w(\gamma) \in V^{\vee}  : w \in W^{J}, \gamma \mbox{ is a weight of $U\otimes \mathbb{C}_{\nu}$} \right\},  
\end{eqnarray}
where $W^J$ is the set of minimal representative in the coset $W/W_J$.
Moreover, the multiplicity of a weight in $X$ coincides with the number of times of the weight appearing in the set (\ref{eqn weight set}).
\end{proposition}

\begin{proof}
We sketch the proof here. Recall that $\Ind^{\mathbb{H}}_{\mathbb{H}_J}U=\mathbb{H} \otimes_{\mathbb{H}_J}U$. By definition, \[\left\{ t_w\otimes u \in w \in W^J \mbox{ and } u \in U  \right\}\] spans the space $\Ind^{\mathbb{H}}_{\mathbb{H}_J}U$. Then we set 
\[F_i= \Span \left\{  t_w \otimes u : w \in W^J \mbox{ and } l(w)\leq i \mbox{ and } u \in U \right\} .\]
Then the graded space $Gr(X):= \oplus_{i \in \mathbb{Z}} F_i/F_{i-1}$ have the same weight spaces as $X$. This proves the proposition.

\end{proof}

\begin{lemma} \label{lem structure rigid}
Let $(J, U)$ and $(J, U')$ be in $\Xi_{\mathrm{rig}}$. Then there exists $\mathbb{H}_J$-modules $Y$ and $Y'$ such that $\Res_{\mathbb{H}_J} X(J, U)= U\oplus Y$ and $\Res_{\mathbb{H}_J}X(J, U')=U' \oplus Y'$  as $\mathbb{H}_J$-modules, and
\[  \Ext^i_{\mathbb{H}_J}(U, Y')=0 \quad \mbox{ for all integers $i$ } .\]
\end{lemma}

\begin{proof}
By considering the central characters of the $\mathbb{H}_J$-submodules of $X$ and using Theorem \ref{thm char ext 0}, $X$ can be written as $X=U_1 \oplus Y$, where $Y$ is the maximal $\mathbb{H}_J$-submodule of $X$ with all weights of $U_1$ in $V_J$, and $Y$ is the maximal $\mathbb{H}_J$-submodule with all weights $\gamma$ of $Y$ not in $ V_{J}$. 

We now show that $U_1=U$. According to Proposition \ref{prop weights}, for any weight $\gamma$ of $Y_1$, $\gamma=w(\sum a_{\alpha^{\vee}} \alpha^{\vee})$, where $a_{\alpha^{\vee}}<0$, $w \in W^J$ and $\alpha^{\vee}$ runs for all the simple coroots in $R_J^{\vee}$. Since $w(\alpha^{\vee})>0$ for all simple coroots in $R_J^{\vee}$ and $\gamma \in V_J^{\vee}$, this forces $w(\alpha^{\vee}) \in R_J^{\vee}$. Combining the conditions that $w(\alpha^{\vee})>0$ and $w(\alpha^{\vee}) \in R_J^{\vee}$, we have $w$ sends all the positive coroots in $R_J^{\vee}$ to the positive coroots in $R_J^{\vee}$. Hence, $w$ permutes the simple coroots in $R_J^{\vee}$ and so $w(J)=J$. Now the condition that $X$ is rigid implies that $w=1$. By counting the multiplicity of weights, we have $U_1=U$ as desired. 

Similarly, we get the decomposition $X'=U' \oplus Y'$ for $Y'$ similarly defined as $Y$. By considering the central characters of $U$ and $Y'$ as $\mathbb{H}_J$-modules and using Theorem \ref{thm char ext 0}, we have the last assertion about $\Ext$-groups in the statement.
\end{proof}

\begin{lemma} \label{lem rigid irr}
Let $(J, U, 0) \in \Xi_{\mathrm{rig}}$. Then the rigid module$X(J, U)$ is irreducible. 
\end{lemma}

\begin{proof}
Set $X=X(J, U)$. By Proposition \ref{prop Hermitian form}, $X$ is isomorphic to the direct sum of irreducible $\mathbb{H}$-modules. Now by Frobenius reciprocity and Lemma \ref{lem structure rigid}, 
\[  \Hom_{\mathbb{H}}(X,X)=\Hom_{\mathbb{H}_J}(U, \Res_{\mathbb{H}_J}X)=\Hom_{\mathbb{H}_J}(U, U)=\mathbb{C}. \] 
This implies $X$ is irreducible.
\end{proof}


Form Lemma \ref{lem structure rigid}, we see that the computation of $\Ext$-groups for a rigid module $X(J, U)$ can be reduced to compute the $\Ext$-groups $\Ext_{\mathbb{H}_J}^i(U, U)$. The study for the $\Ext$-groups among discrete series is out of scope from our development. We need the following result from Opdam-Solleveld for Proposition \ref{prop ext group rigid} later:

\begin{theorem}\cite[Theorem 3.8]{OS} \label{thm ext discrete}
Let $U$ and $U'$ be discrete series of $\overline{\mathbb{H}}_J$. Then
\[   \Ext^i_{\overline{\mathbb{H}}_J} (U, U')=\left\{ \begin{array} {rl}
                                                       \mathbb{C} & \quad \mbox{ if $i=0$ and $U \cong U'$ } \\
																											  0         & \quad  \mbox{ otherwise }    \\
																							\end{array}   
																			\right. 
\]
 \end{theorem}

\begin{proof}
Apply the result \cite[Theorem 3.8]{OS} for affine Hecke algebras. The result can be interpreted in the level of the graded affine Hecke algebra by using Lusztig's reduction theorem \cite{Lu} (See the discussions in \cite[Section 6]{So}).
\end{proof}

\begin{example}
We consider the Steinberg module $\mathrm{St}$ of $\mathbb{H}$, which is a one dimensional space $\mathbb{C}x$ with $\mathbb{H}$-action defined by:
\[  t_{s_{\alpha}}.x=-x \quad \mbox{ for $\alpha \in \Delta$},
\]
\[  v.x=\rho(v)x,
\]
where $\rho$ is the half sum of all the positive coroots in $R^{\vee}$. Then $\Res_W\mathrm{St}=\sgn$, the sign representation of $W$. By the projective resolution in Corollary \ref{cor projective resol} and notations in Section \ref{ss theta twisted},
\[   \Ext_{\mathbb{H}}^i(\mathrm{St}, \mathrm{St})= \frac{\ker d^*: \Hom_{W}(\sgn\otimes \wedge^iV, \sgn )\rightarrow \Hom_W(\sgn\otimes \wedge^{i+1}V, \sgn )}{\im d^*: \Hom_{W}(\sgn\otimes \wedge^{i-1}V, \sgn )\rightarrow \Hom_W(\sgn\otimes \wedge^{i}V, \sgn) } .
\]
Recall that the map $d^*$ is determined by the $\mathbb{H}$-module structure of $\mathrm{St}$. It is well-known that $\left\{ \wedge^iV \right\}_{i=0}^{\dim V}$ are irreducible and mutually non-isomorphic $W$-representations. Hence 
\[  \Hom_{W}(\sgn\otimes \wedge^iV, \sgn )=\left\{ \begin{array}{rl}
                                                          \mathbb{C}  & \quad \mbox{ if $i=0$} \\
																													0           & \quad \mbox{ otherwise}
																									\end{array} \right.
		\]
Hence we have $\Ext_{\mathbb{H}}^i(\mathrm{St}, \mathrm{St})=\mathbb{C}$ for $i=0$ and $\Ext_{\mathbb{H}}^i(\mathrm{St}, \mathrm{St})=0$ for $i>0$ as stated in Theorem \ref{thm ext discrete}.
\end{example}

In order to reduce the amount of notation below, for $\mathbb{H}$-module $X, X'$, we simply write $\Hom_{W}(X \otimes \wedge^i V, X')$ for $\Hom_{W}(\Res_W(X) \otimes \wedge^i V, \Res_W(X'))$. Similar notation is also used for $\Hom$ functor for $W_J$-representations.

\begin{notation} \label{not hom subspace}
Let $J \subset \Delta$ and let $U$ and $U'$ be $W_J$-representations. In Proposition \ref{prop ext group rigid} below, we frequently regard the spaces $\Hom_{W_J}( U \otimes \wedge^l V_J \otimes \wedge^{i-l}V_J^{\bot}, U')$ and $\Hom_{W_J}( U \otimes \wedge^l V_J , U')$ as natural subspaces of $\Hom_{W_J}(U \otimes \wedge^i V, U')$ and $\Hom_{W_J}(U \otimes \wedge^l V, U')$ respectively. In Lemma \ref{lem theta action}, $\Hom_{W_J}(U \otimes \wedge^i V_J^{\bot}, U)$ is regarded as a natural subspace of $\Hom_{W_J}(U \otimes \wedge^l V, U)$.
\end{notation}

\begin{proposition} \label{prop ext group rigid}
Let $(J, U), (J, U') \in \Xi_{\mathrm{rig}}$. Then 
\[ \dim \Ext_{\mathbb{H}}^i(X(J, U),X(J, U'))= \left\{ \begin{array}{cl l}
&\left( \begin{array}{rl} r \\ i \end{array} \right)=\frac{r!}{(r-i)!i!}  & \quad \mbox{ if $U \cong U'$ and $ i \leq r$} \\
&             0                         & \quad \mbox{ otherwise. }
\end{array} \right.
\]
where $r=\dim V-\dim V_J$.
\end{proposition}

\begin{proof}
Let $X=X(J,U)$ and $X'=X(J, U')$. By Lemma \ref{lem structure rigid} and Frobenius reciprocity, $\Ext^i_{\mathbb{H}}(X,X') =\Ext^i_{\mathbb{H}_J}(U, U' \oplus Y') = \Ext^i_{\mathbb{H}_J}(U, U')$, where $Y'$ is an $\mathbb{H}_J$-module as in Lemma \ref{lem structure rigid}. We write $V=V_J \oplus V_J^{\bot}$. For notational convenience, we shall simply write $U$ for $\Res_{W_J}(U)$ below, which should not cause confusion. 

We now apply the projective resolution in (\ref{eqn projective resolution}) on the graded Hecke algebra $\mathbb{H}_J$ which have the root datum $(R_J, V_0, R_J^{\vee}, V_0^{\vee})$ and use $d_{i,U}^*$ for the corresponding differential map as in (\ref{eqn di* actin 1}) and (\ref{eqn di* actin 2}). Note that we could decompose the space
\begin{eqnarray}
\label{eqn hom wj1}    \Hom_{W_J}(U \otimes \wedge^iV, U') &=&\bigoplus_{l=0}^i \Hom_{W_J}( U \otimes \wedge^l V_J \otimes \wedge^{i-l}V_J^{\bot}, U') \\
\label{eqn hom wj2} 		                                    &=&\bigoplus_{l=0}^i a_{r, i, l}\Hom_{W_J}( U \otimes \wedge^l V_J , U'),
\end{eqnarray}
where $a_{r, i, l}=C^r_{i-l}$ if $i-l \leq r$ and  $a_{r, i, l}=0$ if $i-l > r$. Under the above isomorphism, the map $d_{i,U}^*$ and can be in turn expressed as
\[   \bigoplus_{l=0}^i d^*_{l,U}: \Hom_{W_J}(U \otimes \wedge^l V_J, U') \rightarrow \Hom_{W_J}(U \otimes \wedge^{l+1} V_J, U'),
\]
where $\Hom_{W_J}(U \otimes \wedge^l V_J, U')$ and $\Hom_{W_J}(U \otimes \wedge^{l+1} V_J, U')$ are regarded as subspaces of $\Hom_{W_J}(U \otimes \wedge^l V, U')$ and $\Hom_{W_J}(U \otimes \wedge^{l+1} V, U')$ and by abuse of notation, $d_{i,U}^*$ are the maps restricted to the subspaces.
Then the $\Ext$-groups can be expressed as 
\begin{eqnarray} 
& &\Ext_{\mathbb{H}}^i(X, X') \\
 \label{eq ext ker im J}  &=&\bigoplus_{l=0}^ia_{r, i, l}\frac{\ker(d_{l, U}^*: \Hom_{W_J}(U \otimes \wedge^l V_J, U') \rightarrow \Hom_{W_J}(U \otimes \wedge^{l+1} V_J, U')  ) }{\im(d_{l-1, U}^*: \Hom_{W_J}(U \otimes \wedge^{l-1} V_J, U') \rightarrow \Hom_{W_J}(u \otimes \wedge^{l} V_J, U')) }.
\end{eqnarray} 
Then we have 
\begin{eqnarray} \label{eqn ext discrete}   \Ext_{\mathbb{H}}^i(X,X')= \bigoplus_{l=0}^i a_{r, i, l}\Ext^l_{\overline{\mathbb{H}}_J}(U,U') .
\end{eqnarray}
By Theorem \ref{thm ext discrete}, we obtain the statement.

\end{proof}

\subsection{$\theta^*$-action on $\Ext$-groups of rigid modules}

This subsection is devoted to compute the $\theta$-action on $\Ext$-groups of rigid modules.

Let $(J,U, 0) \in \Xi_{\mathrm{rig}}$. Define an $\overline{\mathbb{H}}_{\theta(J)}$-module $U^{\theta}$ such that $U^{\theta}$ is identified with $U$ as vector spaces and the $\overline{\mathbb{H}}_{\theta(J)}$-module structure is determined by: for $u \in U$,
\[    \pi_{U^{\theta}}(t_w)u = \pi_{U}(\theta(t_w))u , \mbox{ for $w \in W_{\theta(J)}$}\]
\[    \pi_{U^{\theta}}(v)u= \pi_U(\theta(v))u, \mbox{ for $v \in V$}. \]

\begin{lemma} \label{lem theta defining action}
Let $(J, U,0) \in \Xi_{\mathrm{rig}}$. Then $X(\theta(J), U^{\theta})$ and $X(J, U)$ are isomorphic.
\end{lemma}

\begin{proof}
Set $X=X(J, U)$. By Corollary \ref{cor hermitian criteria} and Proposition \ref{prop Hermitian form}, $X^{\theta}$ and $X$ are isomorphic. This implies $\Hom_{\mathbb{H}_{\theta(J)}}(U^{\theta} \otimes \mathbb{C}_0, X) \neq 0$. Then the irreducibility of $X$ in Lemma \ref{lem rigid irr} and Frobenius reciprocity implies the statement.

\end{proof}

By Lemma \ref{lem theta defining action}, $\mathbb{H} \otimes_{\mathbb{H}_J} U \cong \mathbb{H} \otimes_{\mathbb{H}_{\theta(J)}} U^{\theta}$ via a map denoted $T_{(J, U)}$. We also define another map $T_{\theta}: \mathbb{H} \otimes_{\mathbb{H}_J}U\rightarrow  \mathbb{H} \otimes_{\mathbb{H}_{\theta(J)}}U^{\theta}  $ given by $\theta(h) \otimes u \mapsto h \otimes u$. Then the map $ T_{(J,U)}^{-1} \circ T_{\theta}$ defines an $\theta$-action on on $\mathbb{H}\otimes_{\mathbb{H}_J} U$ and gives an $\mathbb{H} \rtimes \langle \theta \rangle$-structure on $\mathbb{H}\otimes_{\mathbb{H}_J} U$. Then we see that for any $x \in \mathbb{H}\otimes_{\mathbb{H}_J} U$, $x$ can be uniquely written as the linear combination of
\begin{eqnarray*}
\label{eqn unique expression}
   x=\sum_{w \in W^{\theta(J)}}  t_w \theta(u_w),
\end{eqnarray*}
for some $u_w \in U$.




Recall from Section \ref{ss basic notation} that for $J \subset \Delta$, $w_0^J$ denotes the longest element in $W^J$.

\begin{lemma} \label{lem structure Y 2}
Let $X$, $U$ and $Y$ be as in Lemma \ref{lem structure rigid}. Regard $U$ and $Y$ as subspaces of $X$ (see the proof of Lemma \ref{lem structure rigid}). Then 
\begin{enumerate}
\item Fix a choice of an involution $\theta_J$ on $U$ induced from the longest element in $W_J$. For any non-zero vector $u \in U$, there exists a non-zero scalar $a$ such that $u$ can be uniquely written as 
 \[  \theta_J(u)= a  t_{w_{0}^{\theta(J)}}\theta(u)+\sum_{w  \in W^{\theta(J)} \setminus \left\{ w_{0}^{\theta(J)} \right\}} t_w \theta(u_w) \]
for some $u_{w} \in U$.  (Different choice of the $\theta_J$ action changes the sign of the scalar $a$). 
\item $Y$ is the linear subspace of $X$ spanned by all vectors of the form
\begin{equation}\label{eqn vector form}         t_w \theta(u), \mbox{ for $u \in U$ and for $w \in W^{\theta(J)} \setminus \left\{ w_{0}^{\theta(J)} \right\}$ }. 
\end{equation}

\end{enumerate}

\end{lemma}

\begin{proof}
We define $Y'$ to be the subspace of $X$ spanned by all vectors of the form $t_w\theta(u)$ for $w \in W^J$ and $u \in U$. Then there is a natural projection map $\mathrm{pr}: U \hookrightarrow X \rightarrow X/Y'$. Note that any generalized weight vector of the form 
\[  t_{w_0^{\theta(J)}}\theta(u_w)+y , \mbox{ for $y \in Y'$}\]
has a weight $\theta(w_0^{\theta(J)}(\gamma))=-w_{0,J}(\gamma)$ for some $\gamma \in V_J$. Then by the definition of non-$\theta$-induced and using similar argument as in the proof of Lemma \ref{lem structure rigid}, any generalized weight vector of $X$ lies in $Y'$ does not have a weight in $V_J$. Hence $U \cap Y'=0$ and by considering the dimension, the map $\pr$ is a linear isomorphism. Using the uniqueness of expression in (\ref{eqn unique expression}), we have a map $f$ from $U$ to $U$ such that 
\[    \theta_J(u) = t_{w_0^{\theta(J)}}\theta(f(u)) +y , \mbox{ for $y \in Y'$}.\]
We shall show that $f \circ \theta_J$ is an $\mathbb{H}_J$-module isomorphism. 

We next prove that $Y'$ is invariant under $w \in W_J$. It suffices to show that for any $w \in W_J$, $ww_0^{\theta(J)}W_{\theta(J)}=w_0^{\theta(J)}W_{\theta(J)}$ as cosets. Indeed this follows from 
\[ww_0^{\theta(J)}W_{\theta(J)}=ww_{0}w_{0,\theta(J)}W_{\theta(J)}=w_0\theta(w)w_{0,\theta(J)}W_{\theta(J)}=w_0W_{\theta(J)}.\]

Note that we also have that $Y'$ is invariant under the action of $S(V)$. Hence $Y$ is an $\mathbb{H}_J$-module. 

Now by using the uniqueness property in Lemma \ref{eqn unique expression} with some computations, one can show that $f \circ \theta_J(t_w.u)=t_w.f \circ \theta_J(u)$ for $w \in W_J$ and $f \circ \theta_J(v.u)=v.f\circ \theta_J(u)$. This proves the claim that $f \circ \theta_J$ is an $\mathbb{H}_J$-module isomorphism and Hence $f=a\theta_J$ for some nonzero scalar $a$. This proves (1).

Note that by our description of $Y$ in the proof of Lemma \ref{lem structure rigid} and the fact that any generalized weight vector of $Y'$ does not have a weight in $V_J$, we have $Y=Y'$.


\end{proof}

Let $X=X(J, U)$ be a rigid module. Lemma \ref{lem theta action}(1) below shows $\Ext_{\mathbb{H}}^i(X, X)$ can be identified with a subspace of $\Hom_{W_J}(U \otimes \wedge^i V_J^{\bot}, U)$. Recall that the $\theta^*$-action on $\Ext_{\mathbb{H}}^i(X, X)$ is defined in Section \ref{ss theta twisted}. However, there is no natural way to define a corresponding action of $\theta^*$ on $\Hom_{W_J}(U \otimes \wedge^iV_J^{\bot},U)$ in general. Thus for $\psi \in \Hom_{W_J}(U \otimes \wedge^iV_J, U)$, we define $\overline{\psi} \in \Hom_W(X \otimes \wedge^i V, X)$ such that 
\[ \overline{\psi}((t_w.u) \otimes (v_1 \wedge \ldots \wedge v_i)) =t_w\psi(u \otimes (w^{-1}(v_1)\wedge \ldots \wedge w^{-1}(v_i))) \] for any $w \in W$ and $u \in U$. Here we regard $U$ as a natural subspace of $X \cong \mathbb{H} \otimes_{\mathbb{H}_J}U$ by sending $u$ to $1 \otimes u$. 

\begin{lemma} \label{lem theta action}
Let $X=X(J, U)$ be a rigid module. Regard $U$ as a natural subspace of $X \cong \mathbb{H} \otimes_{\mathbb{H}_J}U$. Let  
\[  d_i^*: \Hom_W(X \otimes \wedge^i V, X) \rightarrow \Hom_W(X \otimes \wedge^{i+1} V, X) \]
and 
\[ d_{i, U}^*: \Hom_{W_J}(U \otimes \wedge^i V, U) \rightarrow \Hom_{W_J}(U \otimes \wedge^{i+1} V, U) \]
be the differential maps for the $\mathbb{H}$-module $X$ and the $\mathbb{H}_J$-module $U \otimes \mathbb{C}_0$ given by (\ref{eqn di* actin 2}). 
\begin{enumerate}
\item The map $\psi \mapsto \overline{\psi}$ induces an isomorphism between the complexes $\left\{ d_{i,U}^*, \Hom_{W_J}(U \otimes \wedge^i V, U) \right\}$ and $\left\{ d_i^*, \Hom_W(X \otimes \wedge^i V, X) \right\}$. The inverse map is given by the map restricting $X \otimes \wedge^i V$ to $U \otimes \wedge^i V$ (as $W_J$-representations).
\item Define $d_i^{U, *}$ to be the restriction of $d_{i,U}^*$ to the subspace $\Hom_{W_J}(U \otimes \wedge^i V_J^{\bot}, U)$ (see notation \ref{not hom subspace}). Then $\Ext^i_{\mathbb{H}}(X, X)$ can be identified with $\ker d_{i}^{U,*}$.
\item We use the identification in (2). For any $\psi \in \Ext^i_{\mathbb{H}}(X, X) \subset \Hom_{W_J}(U \otimes \wedge^i V_J^{\bot}, U)$, $\psi$ is the multiplication of a scalar in the following sense:

 for each fixed $ v_1 \wedge \ldots \wedge v_i \in  \wedge^i V_J^{\bot}$, there exists a scalar $\lambda_{v_1 \wedge \ldots \wedge v_i}$ such that $\psi(u \otimes v_1 \wedge \ldots \wedge v_i)=\lambda_{v_1 \wedge \ldots \wedge v_i} u$ for all $u \in U$.
\item We use the identification in (2). For any $\psi \in \ker d_i^{U,*}$, the map $\theta^*(\overline{\psi})$ is equal to $(-1)^i\overline{\psi} + \phi$ for some $\phi \in \im d^*_{i-1}$.  
\end{enumerate}
\end{lemma}

\begin{proof}

Express $X=U \oplus Y$ as in Lemma \ref{lem structure rigid}. Note that the natural inclusion $U \hookrightarrow \mathbb{H}_{\mathbb{H}_J}U \cong X$ coincides with the natural inclusion $U \hookrightarrow U \oplus Y \cong X$.

We consider (1). As $W_J$-representations, $\Res_{W_J}X= \mathbb{C}[W] \otimes_{\mathbb{C}[W_J]} \Res_{W_J} U$. (1) follows from the Frobenius reciprocity and the fact that $\Ext_{\mathbb{H}_J}(U, Y)=0$ in Lemma \ref{lem structure rigid}.

(2) is implicitly proved in Proposition \ref{prop ext group rigid}. Indeed the expression follows from the identifications in (\ref{eqn hom wj1}), (\ref{eqn hom wj2}) and  (\ref{eq ext ker im J}). Note that from (\ref{eqn hom wj1}) to (\ref{eqn hom wj2}), we drop $\wedge^i V_J^{\bot}$ because $W_J$ acts trivially on $V_J^{\bot}$. However $\theta$ does not act trivially on $V_J^{\bot}$ and so we recover $V_J^{\bot}$ for the computation of $\theta^*$-action here.

For (3), note that from the proof of Proposition of \ref{prop ext group rigid}, we also have
\[\Ext^i_{\mathbb{H}}(X, X)=\ker d_i^{U,*} \cong \Hom_{\mathbb{H}_J}(U, U)\otimes \wedge^i V_J^{\bot}.\] Then the result follows from the Schur's lemma.

We now prove (4). Pick an element $u \in U$. By Lemma \ref{lem structure Y 2}, $\theta_J(u)= at_{w_{0}^{\theta(J)}} \theta(u)+y$ for some non-zero scalar $a$ and for $y \in Y$. 

Without loss of generality, we pick $\psi$ as in (3). For $v_1 \wedge \ldots \wedge v_i \in \wedge^i V_J^{\bot}$,
\begin{eqnarray*}
& & \theta^*(\overline{\psi})(\theta_J(u) \otimes v_1 \wedge \ldots \wedge v_i) \\
&=& \theta^*(\overline{\psi}) ((at_{w_0^{\theta(J)}} \theta(u)+y)\otimes v_1 \wedge \ldots \wedge v_i)\\
& =& a\theta\overline{\psi} (t_{w_{0}^J} u\otimes \theta(v_1) \wedge \ldots \wedge \theta(v_i))+\theta^*(\overline{\psi}) (y\otimes v_1 \wedge \ldots \wedge v_i)  \\
&=&a t_{ w_{0}^{\theta(J)}}\theta\overline{\psi} (u\otimes (w_0^J)^{-1}\theta(v_1) \wedge \ldots \wedge (w_0^J)^{-1}\theta(v_i)) +\theta^*(\overline{\psi}) (y\otimes v_1 \wedge \ldots \wedge v_i) \\
&=& (-1)^ia t_{ w_{0, \theta(J)}}\lambda_{v_1 \wedge \ldots \wedge v_i}\theta(u)  +\theta^*(\overline{\psi}) (y\otimes v_1 \wedge \ldots \wedge v_i) \quad \mbox{ by (3) } \\
&=&  (-1)^i\lambda_{v_1 \wedge \ldots \wedge v_i} \theta_J(u)- (-1)^i\lambda _{v_1 \wedge \ldots \wedge v_i}y +\theta^*(\overline{\psi}) (y\otimes v_1 \wedge \ldots \wedge v_i) \\
&=&  (-1)^i \overline{\psi}(\theta_J(u) \otimes v_1 \wedge \ldots \wedge v_i)- (-1)^i\lambda_{v_1 \wedge \ldots \wedge v_i} y +\theta^*(\overline{\psi}) (y\otimes v_1 \wedge \ldots \wedge v_i) 
\end{eqnarray*}
We now define $\phi'(\theta_J(u) \otimes v_1 \wedge \ldots \wedge v_i)=- (-1)^i\lambda_{v_1 \wedge \ldots \wedge v_i} y +\theta^*(\overline{\psi}) (y\otimes v_1 \wedge \ldots \wedge v_i)$ if $v_1 \wedge \ldots \wedge v_i \in \wedge^i V_J^{\bot}$ and $\phi'(\theta_J(u) \otimes v_1 \wedge \ldots \wedge v_i)=0$ otherwise. Note that $\theta^*(\overline{\psi}) (y\otimes v_1 \wedge \ldots \wedge v_i)$ is in $Y$ by using Lemma \ref{lem structure Y 2} (2) and hence $\phi' \in \Hom_{W_J}(U\otimes \wedge^iV,  Y)$. Since $\Ext^i_{\mathbb{H}_J}(U, Y)=0$, this implies that $\phi' \in \im d_{i-1}^*$ by definition. Now 
$  \theta^*(\overline{\psi})-(-1)^i \overline{\psi}-\phi' $
is indeed a map lying in the subspace \[\bigoplus_{l=1}^{i}\Hom_{W_J}(U \otimes \wedge^{l}V_J \otimes \wedge^{i-l}V_J^{\bot}, U).\]
This is again in $\im d_{i-1}$ by following some computation in Proposition of \ref{prop ext group rigid} and we omit the detail.  
\end{proof}




\begin{theorem} \label{thm theta action}
Let $\mathbb{H}$ be the graded affine Hecke algebra associated to a crystallographic root system. Let $X=X(J, U)$ and $X'=X(J, U')$ for some $(J, U, 0), (J, U',0) \in \Xi_{\mathrm{rig}}$ (i.e. $X$ and $X'$ are rigid modules (Definition \ref{def rigid})). Then 
\[ \dim \Ext_{\mathbb{H}}^i(X,X')= \left\{ \begin{array}{cll}
&\left( \begin{array}{rl} r \\ i \end{array} \right) =\frac{r!}{(r-i)!i!} & \quad \mbox{ if $U \cong U'$ and $ i \leq r$} \\
&             0                         & \quad \mbox{ otherwise, }
\end{array} \right.
\]
where $r=\dim V-\dim V_J$. $\theta^*$ defined in (\ref{eqn theta* actin 1}) acts by the multiplication of a scalar of $(-1)^i$ on $\Ext^i(X, X')$. 
\end{theorem}

\begin{proof}
The first assertion is Proposition \ref{prop ext group rigid}. For the second assertion, we only have to consider $U'=U$ in view of Proposition \ref{prop ext group rigid}. 
 With Lemma \ref{lem theta action} (1), we rewrite
\[ \Ext_{\mathbb{H}}^i(X, X) = \ker(d_i^{U,*}: \Hom_{W_J}(U \otimes \wedge^i V_J^{\bot}, U) \rightarrow \Hom_{W_J}(U \otimes \wedge^{i+1} V, U)  ),
\]  
 Now using Lemma \ref{lem theta action} (1) and (4), we have that $\theta^*$ acts by $(-1)^i$ on $\Ext_{\mathbb{H}}^i(X,X)$.

\end{proof}

\begin{remark}
The author would like to thank Maarten Solleveld for pointing out \cite[Theorem 5.2]{OS2}.

The  $\Ext$-groups for arbitrary tempered modules can be computed from a simple formula in \cite[Theorem 5.2]{OS2}. In particular, if $X=X(J,U)$ for some $(J, U, 0) \in \Xi$ and $X$ is irreducible, then $\Ext_{\mathbb{H}}^i(X, X) \cong \wedge^i V_J^{\bot}$. However, it seems not to be direct to know the $\theta^*$-action on the $\Ext$-groups from \cite{OS2}. 
\end{remark}

As a consequence of Theorem \ref{thm theta action} and Corollary \ref{cor extended}, we have the following result.

\begin{corollary} \label{cor induced twisted EP}Let $X=X(J,U)$ be a rigid module of discrete series. Set $r =\dim V_J^{\bot}$. Then
\begin{enumerate}
\item  $\mathrm{EP}^{\theta}_{\mathbb{H}}(X, X) =2^r \neq 0$. 
\item $\dim \Ext^i_{\mathbb{H} \rtimes \langle \theta \rangle}(X, X)= \left( \begin{array}{rl} r \\ i \end{array} \right)$ for all even $i$ with $i \leq r$ and $\dim \Ext^i_{\mathbb{H} \rtimes \langle \theta \rangle}(X, X)=0$ otherwise. 
\end{enumerate}
\end{corollary}


There is another application of the twisted Euler-Poincar\'e pairing for the deformation or complementary series of rigid modules. 

\begin{corollary} \label{cor deform} (c.f  \cite[Remark 4.6]{BCT})
For each $(J, U, \nu) \in \Xi$, set $X_{\nu}=X(J, U, \nu)$. Assume $X_0$ satisfies one of the three conditions in Proposition \ref{prop herm}.
\begin{enumerate}
\item There exists a non-zero $\nu \in (V_J^{\vee})^{\bot}$ such that $\Res_WX_{\nu} \cong \Res_WX_{\nu}^{\theta}$ only if $X_0$ is a rigid module.
\item There exists a non-zero $\nu \in (V_J^{\vee})^{\bot} \cap V^{\vee}_0$ such that $X_{\nu}$ is $*$-Hermitian only if $X_0$ is a rigid module.
\end{enumerate}
\end{corollary}


\begin{proof}
Suppose $\Res_WX_{\nu} \cong \Res_WX_{\nu}^{\theta}$ for some non-zero $\nu \in (V_J^{\vee})^{\bot}$. Then by considering the central characters of the modules and using Theorem \ref{thm char ext 0}, $\mathrm{EP}^{\theta}_{\mathbb{H}}(X_0, X_{\nu})=0$. Then by Theorem \ref{thm twisted ext}, $\mathrm{EP}^{\theta}_{\mathbb{H}}(X_0, X_0)=0$. Hence, $X_0$ is not a rigid module by Corollary \ref{cor induced twisted EP}. This proves (1). For (2), it follows from (1) and Proposition \ref{prop herm}.


\end{proof}

\begin{example}
The result for Theorem \ref{thm theta action} is not true for other parabolically induced modules in general. For instance, consider $\mathbb{H}$ of type $A_2$. Take $J=\emptyset$. Let $U$ be the one-dimensional trivial representation of $\overline{\mathbb{H}}_{\emptyset}=\mathbb{C}$ and let $X=X(\emptyset, U)$. Then $X(\emptyset, U)$ is an irreducible parabolically induced module of $\mathbb{H}$. Direct computation using Frobenius reciprocity shows 
\[ \dim \Ext_{\mathbb{H}}^i(X,X)= \left\{ \begin{array}{rll}
&1 & \quad \mbox{ if $i=0,2$} \\
& 2                               & \quad \mbox{ if $i=1$ } \\
& 0          & \quad \mbox{ if $i \geq 3$}
\end{array} \right.
\]
Moreover, $\theta^*$ acts as an identity on $\Ext_{\mathbb{H}}^0(X,X)$, acts as the diagonal matrix $\diag(1,-1)$ on $\Ext_{\mathbb{H}}^1(X,X)$ and acts as $-1$ on $\Ext_{\mathbb{H}}^2(X,X)$.
\end{example}





\section{Solvable tempered modules and twisted elliptic spaces} \label{s solvable induced}

The goal of this section is to put or recollect some results in \cite{Ci}, \cite{CH0}, \cite{CH}, \cite{Re} and \cite{OS2} in the framework of twisted elliptic spaces.  

\subsection{Kazhdan-Lusztig model}

In this section, let $\mathbb{H}$ be the graded affine Hecke algebras associated to a crystallographic root datum $(R, V, R^{\vee}, V^{\vee})$ and an equal parameter function $k\equiv 1$. We also assume $R$ spans $V$. Let $\mathfrak{g}$ be the Lie algebra of the corresponding type. Let $G$ be the simply-connected Lie group associated to $\mathfrak{g}$.  According to the Kazhdan-Lusztig parametrization, there is a one-to-one correspondence between the set of irreducible tempered modules $X(e, \phi)$ with real central characters and the $G$-orbits of the set 
\[  \left\{ (e, \phi): e\in \mathcal N, \phi \in \widehat{A(e)}_0 \right\}, \]
where $\mathcal N$ is the set of nilpotent elements in $\mathfrak{g}$, $A(e)$ is the component group of $e$ and $\widehat{A(e)}_0$ is the set of irreducible representation of the component group $A(e)$ that appears in the Springer correspondence.

We define $\mathcal N_{\mathrm{sol}}$ to be the set of nilpotent elements with a solvable centralizer in $\mathfrak{g}$. The interest for the set $\mathcal N_{\mathrm{sol}}$ can be found in \cite{Ci}, \cite{BCT}, \cite{Ch1} and \cite{CH}. We shall use the Bala-Carter symbols for the nilpotent orbits.

\begin{definition}
We say an irreducible tempered module $X(e, \phi)$ (with a real central character) is solvable if $e \in \mathcal N_{\mathrm{sol}}$. 
\end{definition}

 We need to use the following fact in the Kazhdan-Lusztig model \cite[6.2]{KL} (also see \cite[6.1a]{Re}): 

\begin{lemma} \label{lem induce kl model}
 Let $e$ be a nilpotent element and let $L$ be a Levi subgroup of $G$ containing $e$. Let $J$ be the subset of $\Delta$ associated to $L$ and let $A_L(e)$ be the component group of $e$ in $L$. Then for an $A_L(e)$-representation $\phi$, denote $U_J(e, \phi)$ the tempered $\overline{\mathbb{H}}_J$-module associated to the pair $(e, \phi)$ in the Kazhdan-Lusztig model. Let $X_J(e, \phi)= U_J(e, \phi) \otimes \mathbb{C}_0$ be an $\mathbb{H}_J \cong \overline{\mathbb{H}}_J \otimes S(V_J^{\bot})$-module. Then
\[   \Ind_{\mathbb{H}_J}^{\mathbb{H}} X_J(e, \phi) =X(e, \Ind_{A_L(e)}^{A(e)} \phi) .\]
\end{lemma}




\subsection{Dimension of twisted elliptic spaces}
For Theorem \ref{thm dimension} below, we apply the Kazhdan-Lusztig model to study the twisted elliptic spaces for non-trivial $\theta$. Anyway, we shall use \cite[Theorem 6.4]{OS2} when $\theta$ is trivial and also apply some computations in \cite{Ci}. Perhaps one may also apply \cite[Theorem 6.4]{OS2} or its line of argument to obtain Theorem \ref{thm dimension} below in general. 

\begin{theorem} \label{thm dimension}
Let $\mathbb{H}$ be a graded affine Hecke algebra associated to a crystallographic root system and an arbitrary parameter function $k$. The dimension of $\mathrm{Ell}_{\mathbb{H}}^{\theta}$ is equal to the number of $\theta$-twisted elliptic conjugacy classes.
\end{theorem}

\begin{proof}

 For $\theta=\mathrm{Id}$, it follows from \cite[Theorem 6.4]{OS2} (in more detail, one also has to apply \cite[Proposition 6.4]{So}). For $\theta \neq \mathrm{Id}$, if $k_{\alpha}=0$ for all $\alpha \in \Delta$., it is easy by Theorem \ref{thm twisted ext}. Thus we only consider the case that the parameter function $k_{\alpha} \neq 0$ for all $\alpha \in \Delta$. It is well-known that $\Res_WX(e, \phi)$ (for all $e \in \mathcal N$ and $\phi \in \widehat{A(e)}_0$) spans the representation ring of $W$. Then the dimension of the spanning set of $\left\{ \Res_WX(e,\phi) \otimes S : e \in \mathcal N, \phi \in \widehat{A(e)}_0 \right\}$ is equal to the number of twisted ellitpic conjugacy classes. The last statemenet follows from a case-by-case analysis. The dimension of the spanning set follows from \cite[Theorem 1.0.1]{Ci}. The number of $\theta$-twisted elliptic conjugacy classes is as follows:
\[ A_n: \mbox{ number of partitions of $n$ with distinct parts}, \]
\[ D_n\ \mbox{($n$ odd)}: \mbox{ number of partitions of $n$ with odd number of parts},\ E_6:\ 9 . \]
Now by Theorem \ref{thm twisted ext} and Proposition \ref{prop ep dirac index}, we obtain that $\mathrm{dim} \mathrm{Ell}_{\mathbb{H}}^{\theta}$ is equal to the number of $\theta$-twisted elliptic conjugacy classes.
\end{proof}

\subsection{Description for twisted elliptic spaces}

\begin{theorem} \cite{CH}\label{cor twisted elliptic}
Let $\mathbb{H}$ be a graded affine Hecke algebra associated to a crystallographic root system and an equal parameter function $k \equiv 1$. Then
\begin{enumerate}
\item $\mathrm{EP}^{\theta}_{\mathbb{H}}(X(e, \phi), X(e, \phi)) \neq 0$ for any $\phi \in \widehat{A(e)}_0$ if and only if $e \in \mathcal N_{\mathrm{sol}}$.
\item $\mathrm{EP}^{\theta}_{\mathbb{H}}(X(e, \phi), X(e', \phi')) = 0$ if $e$ and $e'$ are not in the same nilpotent orbit.
\item The set $\left\{ [X(e, \phi)]: e \in \mathcal N_{\mathrm{sol}}, \phi \in \widehat{A(e)}_0 \right\}$ spans the $\theta$-twisted elliptic space $\mathrm{Ell}_{\mathbb{H}}^{\theta}$. 
\end{enumerate}
\end{theorem}

\begin{proof}

For (1) and (2), this is a direct consequence of Theorem \ref{thm twisted ext} and results in \cite[Theorem 1.1, Theorem 1.3]{CH}. For (3), it follows from (1) and the fact that $X(e, \phi)$ (for all nilpotent element $e$ and all $\phi \in \widehat{A(e)}_0$) span the entire representation ring of $W$. From (1), we know that for $e \notin \mathcal N_{\mathrm{sol}}$, $X(e, \phi)$ has a zero image in $\mathrm{Ell}_{\mathbb{H}}^{\theta}$. Hence, the set in (3) spans the space $\mathrm{Ell}_{\mathbb{H}}^{\theta}$. 

We remark that for (2), one can also prove directly by considering the central characters of those modules. In more detail, the central character of $X(e, \phi)$ is $\frac{1}{2}h_e$, where $h_e \in V^{\vee}$ is the semisimple element in the $\mathfrak{sl}_2$-triple $\left\{ e, h_e, f \right\}$. If two nilpotent elements $e$ and $e'$ are not in the same nilpotent orbit, then the two elements $h_e$ and $h_{e'}$ are not in the same $W$-orbit in $V^{\vee}$ (\cite[Theorem 2.2.4]{CM}, \cite[Theorem 3.2.14]{CM}). 

In the case of type $A_n$, all solvable tempered modules are rigid (see the proof of Proposition \ref{prop sol rigid} below). Thus for type $A_n$, (1) and (3) can also be obtained by Corollary \ref{cor induced twisted EP} and a simple argument using Theorem \ref{thm dimension} and using (2).





\end{proof}
\begin{remark}
For arbitrary parameters, we expect some similar results as Theorem \ref{cor twisted elliptic} can be obtained by considering tempered modules of solvable central characters. Here solvable central characters are in the sense of \cite{Ch1}. 
\end{remark}

\subsection{Relation between rigid modules and solvable tempered modules}
We extend the notation of $X(e, \phi)$ to any $A(e)$-representation $\phi$: define 
\[ X(e, \phi)=\bigoplus_{\phi' \in \widehat{A(e)}_0} m_{ \phi'}\ X(e,\phi'), \]
where $m_{ \phi'}=\dim \Hom_{A(e)}(\phi', \phi)$.
 
\begin{proposition} \label{prop sol rigid}
Let $\mathbb{H}$ be of type $A_n$, $D_n$ ($n$ odd) and $E_6$. Let $X$ be a parabolically induced tempered module with a real central character. Then $X$ is solvable and irreducible if and only if $X$ is rigid.
\end{proposition}

\begin{proof}

This is a case-by-case analysis. To check which nilpotent orbits lie inside $\mathcal N_{\mathrm{sol}}$, one may use the description of the centralizer of a nilpotent element in \cite[Chapter 13]{Ca} (also see \cite{Ci}) (one may also verify by using the combinatorial criteria given in \cite[Definition 1.1]{Ch1}). 

For type $A_n$, a nilpotent element is in $\mathcal N_{\mathrm{sol}}$ if and only if the Jordan canonical form of $e$ has blocks of distinct sizes. The Bala-Carter symbols for nilpotent elements in $\mathcal N_{\mathrm{sol}}$ coincide with the list for type $A_n$ (Remark \ref{rmk classify rigid}). Furthermore, for type $A_n$, all $X(e,\phi)$ for any $\phi \in \widehat{A(e)}_0$ are irreducible and hence the statement for type $A_n$ is clear.

For type $E_6$, a nilpotent element is in $\mathcal N_{\mathrm{sol}}$ if and only if the Bala-Carter symbol for the nilpotent element is of type $E_6$, $E_6(a_1)$, $E_6(a_3)$, $D_5$, $D_5(a_1)$, $A_4+A_1$ and $D_4(a_1)$. The only type that does not appear in the classification of rigid modules is type $D_4(a_1)$. By Lemma \ref{lem rigid irr}, we only have to verify in the case that any irreducible tempered module associated to $e$ of type $D_4(a_1)$ is not a parabolically induced module. Note that the corresponding component group $A(e)$ is $S_3$ and all representations of $A(e)$ appear in the Springer correspondence.

Let $e$ be of type $D_4(a_1)$ and $\phi \in \widehat{A(e)}_0$. Suppose $X(e, \phi)=\Ind_{\mathbb{H}_J}^{\mathbb{H}} X_J(e,\phi')$ for some proper $J \subset \Delta$ and some $A_L(e)$-representation $\phi'$. Here we use the notation in Lemma \ref{lem induce kl model}. Note that $J$ can only be of type $D_5$ or $D_4(a_1)$ and the component groups $A_L(e)$ of $e$ for the Levi subgroups corresponding to $D_5$ and $D_4$ are  $S_2$, and $1$ respectively, and hence $\Ind_{A_{L}(e)}^{A(e)} \phi'$ is not a single representation of $S_3$. This contradicts the irreducibility of $X$. Hence $X(e, \phi)$ is not parabolically induced from some discrete series.

We now consider the case of $D_n$ ($n$ odd). In this case, a nilpotent element in $\mathfrak{so}(2n)$ is in $\mathcal N_{\mathrm{sol}}$ if and only if the partition of $e$ contains only odd parts and each odd part has multiplicity at most 2. Then a similar analysis as in the case of $E_6$ will yield the result. In the analysis, we need the following description of the component group of (arbitrary) nilpotent orbits for $\mathfrak{so}(2m)$ for both $m$ odd and even (see for example \cite[Chapter 6]{CM}): $A(e)=(\mathbb{Z}/2\mathbb{Z})^{\mathrm{max}(0,a-1)}$ if all odd parts have even multiplicity, $A(e)=(\mathbb{Z}/2\mathbb{Z})^{\mathrm{max}(0, a-2)}$ otherwise, where $a$ is the number of distinct odd parts in the partition of $e$. We also need the component group of any nilpotent element in $\mathfrak{sl}(p)$ is trivial. Moreover, we also need the fact that for $e \in \mathcal N_{\mathrm{sol}}$, all the representations of $A(e)$ appear in the Springer correspondence. 


\end{proof}

\begin{remark} \label{rmk non ellip rig}
In type $A_n$ and $E_6$, solvable modules which are not elliptic are indeed rigid. However, in type $D_n$ ($n$ odd) with $n \geq 9$, if $e$ is a nilpotent element corresponding to a partition satisfying the following three conditions:
\begin{enumerate}
\item $e$ has no even parts, and
\item $e$ has all odd parts with multiplicity $2$, and 
\item the number of distinct odd parts of $e$ is at least 3, 
\end{enumerate}
then $X(e, \phi)$ is solvable, but neither rigid nor elliptic. 

\end{remark}

\begin{remark} \label{rmk general formulation}
It is also possible to extend the condition of rigid modules to all solvable modules. We expect that an irreducible tempered module $X(e, \phi)$ with a real central character is solvable if and only if $X(e, \phi)$ is a submodule of a parabolically induced module $X(J, U)$ for some $(J, U, 0) \in \Xi$ such that
\[  \card \left\{ w \in W : w(J)=J,\ w(U)=U \right\}
\]
is equal to the sum of the square of the multiplicity of each irreducible submodule in $X(J, U)$. 
\end{remark}

\subsection{Description of the radical of $\mathrm{EP}^{\theta}_{\mathbb{H}}$}
We end this paper with the following description of the radical: 
\begin{conjecture}
The radical $\mathrm{rad}(\mathrm{EP}^{\theta}_{\mathbb{H}})$ in $K_{\mathbb{H}}^{\theta}$ is equal to the image of
\begin{equation*}\label{eqn radical} \bigoplus_{J \in \mathcal J^{\theta} }\Ind_{\mathbb{H_J} \rtimes \langle \theta \rangle}^{\mathbb{H} \rtimes \langle \theta \rangle}    K_{\mathbb{C}}(\mathbb{H}_J \rtimes \langle \theta \rangle) .
\end{equation*}
\end{conjecture}

When $\theta=\Id$, it is known to be true from \cite[Theorem 6.4]{OS2}. It is also possible to apply \cite[Theorem 6.4]{OS2} or its proof for the conjecture in general. For non-trivial $\theta$, it is not too hard to verify directly for type $A_n$ and $E_6$, but it seems more effort has to be done for type $D_n$ ($n$ odd).

\end{document}